\documentclass[11pt]{amsart}

\usepackage{amssymb,amsmath,enumerate,epsfig}
\usepackage{amsfonts,color}
\usepackage{a4,latexsym}
\usepackage{parskip}

\usepackage{hyperref}
\hypersetup{pdftitle=Enriques}

\newcommand{\C} {\mathbb{C}}
\newcommand{\Q} {\mathbb{Q}}
\newcommand{\N}  {\mathbb{N}}
\newcommand{\R} {\mathbb{R}}

\newcommand{\Z}{\mathbb{Z}}

\newcommand{\PP}{\mathbb{P}}
\newcommand{\NS}{\mathop{\rm NS}}
\newcommand{\Num}{\mathop{\rm Num}}
\newcommand{\MW}{\mathop{\rm MW}}
\newcommand{\MWL}{\mathop{\rm MWL}}
\newcommand{\Km}{\mathop{\rm Km}}

\newcommand{\A}{\mathbb{A}}

\newcommand{\Br}{\mathop{\rm Br}\nolimits}

\newcommand{\Aut}{\mathop{\rm Aut}}

\newcommand{\mX}{\mathcal X}

\newcommand{\mY}{\mathcal Y}

\newcommand{\fS}{\mathfrak S}

\newcommand{\ra}{\rightarrow}

\newcommand{\eprf}{\hfill $\square$ \bigskip\par}

\newtheorem{Theorem}{Theorem}
\newtheorem{Proposition}[Theorem]{Proposition}
\newtheorem{Lemma}[Theorem]{Lemma}

\newtheorem{prop}[Theorem]{Proposition}
\newtheorem{lemma}[Theorem]{Lemma}

\newtheorem{Corollary}[Theorem]{Corollary}

\theoremstyle{remark}

\newtheorem{rem}[Theorem]{Remark}

\theoremstyle{definition}

\begin{document}

\title{Enriques Surfaces -- Brauer groups and Kummer structures}

\author{Alice Garbagnati}
\address{Dipartimento di Matematica, Universit\`a di Milano,
  via Saldini 50, I-20133 Milano, Italy}
\email{alice.garbagnati@unimi.it}
\urladdr{http://sites.google.com/site/alicegarbagnati/home}

\author{Matthias Sch\"utt}
\address{Institut f\"ur Algebraische Geometrie, Leibniz Universit\"at
  Hannover, Welfengarten 1, 30167 Hannover, Germany}
\email{schuett@math.uni-hannover.de}
\urladdr{http://www.iag.uni-hannover.de/$\sim$schuett/}

\subjclass[2000]{14J28; 14J27, 14J50, 14F22}

\keywords{Enriques surface, Kummer surface, elliptic fibration, Brauer group, Shioda-Inose structure, Enriques Calabi-Yau}
%


%
\date{February 10, 2011}

 \begin{abstract}
This paper develops families of complex Enriques surfaces whose Brauer groups pull back identically to zero on the covering K3 surfaces.
Our methods rely on isogenies with Kummer surfaces of product type.
We offer both lattice theoretic and geometric constructions.
We also sketch how the construction connects to string theory and Picard--Fuchs equations in the context of Enriques Calabi-Yau threefolds. 
\end{abstract}
 
\maketitle
 
\section{Introduction}


The Brauer group is an important, but very subtle birational invariant of a projective surface.
In \cite{Beau},
Beauville proved that generically the Brauer group of a complex Enriques surface injects into
the Brauer group of the covering K3 surface.
Subsequently Beauville asked for explicit examples 
where the Brauer groups pull back identically to zero.
This problem has recently been solved in \cite{GvG} (see also Section \ref{ss:GvG}) and in \cite{HS},
but only by isolated so-called singular K3 surfaces (Picard number 20).
In this paper we develop methods to derive such surfaces in one-dimensional families.
Our results cover both the Kummer and the non-Kummer case.
In Section \ref{s:lat},
we construct for any integer $N>1$ a one-dimensional family $\mX_N$ of complex K3 surfaces 
with Picard number $\rho\geq 19$
such that the general member admits an Enriques involution $\tau$.

\begin{Theorem}
\label{thm}
Let $N>1$. 
Consider a general K3 surface $X_N\in \mX_N$, i.e.~$\rho(X_N)=19$.
Denote the quotient by the Enriques involution $\tau$ by $Z_N$.
Then
\[
\pi^* \Br(Z_N) =
\begin{cases}
\{0\} & \text{ if $N$ is odd},\\
\Z/2\Z & \text{ if $N$ is even}.
\end{cases}
\]
\end{Theorem}

The K3 surfaces in the family $\mX_N$ are generally not Kummer,
but we exploit a geometric construction related to Kummer surfaces 
of $N$-isogenous elliptic curves in Section \ref{s:geom}.
By similar methods we derive for any $N\in\N$ a one-dimensional family $\mY_N$ 
consisting of  Kummer surfaces with Picard number $\rho\geq 19$
and  Enriques involution $\tau$ in Section \ref{s:Km}.

\begin{Theorem}
\label{thm2}
Let $N\in\N$. 
Consider a Kummer surface $Y_N\in \mY_N$ with~$\rho(Y_N)=19$.
Let $\tau$ denote the Enriques involution on $Y_N$.
Then
\[
\pi^* \Br(Y_N/\tau) =
\begin{cases}
\{0\} & \text{ if $N$ is odd},\\
\Z/2\Z & \text{ if $N$ is even}.
\end{cases}
\]
\end{Theorem}

The two theorems show that the Enriques surfaces in question
come in families.
We shall work out one family in detail in Theorem \ref{thm:6}.
The general assumption that the K3 surfaces have non-maximal Picard number $\rho=19$ is 
fairly mild and not strictly necessary (see Proposition \ref{prop:spec} and  Section \ref{ss:comments}).

Our main construction uses elliptic fibrations and isogenies of  K3 surfaces, sometimes
in the context of  Shioda-Inose structures.
The afore-mentioned relation with Kummer surfaces of product type
also give rises to applications of our techniques to string theory
and Picard-Fuchs equations (see Section \ref{s: BP elliptic fibr and moduli}).

The paper is organised as follows.
In the next section we review the basic properties of K3 surfaces and Enriques surfaces 
relevant to this work. 
We also discuss Beauville's result on the pull-back of the Brauer group
and the subsequent problem posed by him (\ref{ss: Beauville's problem}).
In Section \ref{s:lat} we construct the families of K3 surfaces $\mX_N$ with Enriques--involution
in Theorem \ref{thm} from a lattice-theoretic point of view. These families will be constructed as specialisation of the Barth--Peters 2-dimensional family of K3 surfaces (\cite{BP}). The geometric properties and the moduli space of the Barth--Peters family will be described in Section \ref{s: BP elliptic fibr and moduli}. 
We also point out relations to other Enriques surfaces, 
and in particular to Calabi-Yau threefolds and their Picard-Fuchs equations investigated in string theory.
Section \ref{s:geom} is the geometric center of the paper: 
the families $\mX_N$ are reconsidered and constructed in a very geometric way.
As an illustration, the family $\mX_3$  is analysed in Section \ref{s:N=3}.
We give explicit equations over $\Q$ and describe the specialisations of this family to singular K3 surfaces, relating them with the presence of complex multiplication on certain elliptic curves. In Section \ref{s:Km} we introduce the Kummer families $\mY_N$ and prove Theorem \ref{thm2} by techniques similar to the ones applied in previous sections to the families $\mX_N$.


\section{Basic Properties}
\label{s:basics}

Throughout this paper we work over number fields (such as the field of rational numbers $\Q$) or over the field of complex numbers $\C$.
For basic properties of K3 surfaces and Enriques surfaces relevant to our paper, we refer to \cite{HS} and the references therein, in particular \cite{BHPV}.
Here we concentrate on the most important ingredients for this paper.
The main motivation for this work is explained in
\ref{ss: Beauville's problem} 
where we discuss Beauville's result on the pull-back of the Brauer group
and the subsequent problem posed by him.

\subsection{Lattices}
A lattice is a pair $(L,b)$ where $L$ is a free $\Z$-module  of finite dimension and $b$ is a symmetric bilinear form defined over $L$ and taking values in $\Z$. We often omit $b$, when the bilinear form is clear by the context. If $(L,b)$ is a lattice, we  denote by $L(n)$, $n\in\Z$, the same $\Z$-module with the bilinear form multiplied by $n$. We will denote by $nL$ the lattice obtained as direct sum of $L$ $n$-times.\\
We denote by $L^{\vee}:=Hom(L,\Z)$ the dual lattice of $(L,b)$. 
The discriminant lattice of $(L,b)$ is defined as $A_L:=L^{\vee}/L$ 
and we denote by $l(A_L)$ the minimum number of its generators. 
The bilinear form $b$ induces a bilinear form on $L^{\vee}/L$ taking values in $\Q\mod\Z$, 
which is called discriminant form. 
The signature $(s_+,s_-)$ of a lattice $(L,b)$ is 
the signature of the $\R$-linear extension of the bilinear form $b$. 
A lattice $(L,b)$ is said to be even if $b(l,l)\in 2\Z$ for each $l\in L$, 
and it is said to be unimodular if $L\simeq L^{\vee}$. 
We recall the following results, due to Nikulin,  that will be used in the sequel.
\begin{prop}\label{prop: uniquely of a lattice in primitve lattice, condition on discriminant}{\rm \cite[Corollary 1.13.3]{Nikulin bilinear}} Let
$L$ be an even lattice with signature $(s_{+},s_{-})$ and
discriminant form $q_L$. If $s_{+}>0$, $s_{-}>0$, and
$l(A_L)\leq rank(L)-2$, 
then up to isometry $L$ is the only lattice
with these invariants.
\end{prop}
\begin{Theorem}\label{theorem: Nikulin 1.14.4}{\rm \cite[Theorem 1.14.4]{Nikulin bilinear}} Let
$M$ be an even lattice with invariants $(t_+, t_-, q_M)$ and $L$
be an even unimodular lattice of signature $(s_+,s_-)$. Suppose that
$$t_+<s_+,\ \ \ t_-<s_-,\ \ \ l(A_M)\leq \mbox{rank}(L)-\mbox{rank}(M)-2.$$
Then there exists a unique primitive embedding of $M$ in $L$.
\end{Theorem}

The second cohomology group of a K3 surface $W$ with integer coefficients, $H^2(W,\Z)$, with the pairing induced by the
cup product is a lattice isometric to the even unimodular lattice
$\Lambda_{K3}:=2E_8(-1)+ 3U$ (the K3 lattice), where $U$ is
the hyperbolic rank 2 lattice with pairing
$\left[\begin{array}{rr}0&1\\1&0\end{array}\right]$ and $E_8(-1)$
is the rank 8 negative definite lattice associated to the Dynkin diagram $E_8$ (cf.~\cite{BHPV}). The N\'eron Severi group of $W$, $\NS(W)=H^{1,1}(W)\cap H^2(W,\Z)$, is a sublattice of $H^2(W,\Z)$
the rank of which is called Picard number and denoted by $\rho(W)$. The transcendental lattice of $W$, denoted by $T(W)$, is the orthogonal complement of $\NS(W)$ in the lattice $H^2(W,\Z)$.

\subsection{Lattice enhancements}
\label{construction: specialisation} 
\label{constr}
We explain a lattice-theoretic method in order to determine 
certain subfamilies of a family of K3 surfaces with given transcendental lattice.
The ideas are based on the theory of lattice--polarised K3 surfaces.

Let $W$ be the generic member of a family of K3 surfaces with transcendental lattice $T(W)$.
Here $\rho(W)<20$ as otherwise the family would solely consist of $W$.
Hence $T(W)$ is indefinite,
and we can fix a vector $v$ in $T(W)$ with negative self--intersection. 
By the surjectivity of the period map, there exists a K3 surface $W_2$, such that 
\[
T(W_2)\simeq v^{\perp_{T(W)}}.
\]
The surface $W_2$ is the general member of a subfamily of the family of $W$ of codimension one. 
Clearly the N\'eron--Severi lattice of  $W_2$ is an overlattice of finite index of $\NS(W)+ \langle v\rangle$. 
This can be made precise as follows:
$\NS(W_2)$ is the minimal primitive sublattice of $H^2(W,\Z)\simeq H^2(W_2,\Z)$ containing $\NS(W)+\langle v\rangle$. 

We explain how to find a $\Z$-basis for $\NS(W_2)$.
Recall the following connection between the discriminant forms of $T(W)$ and $\NS(W)$.
Let  $n_i\in \NS(W)$, $d_i\in \mathbb{N}$ 
such that $n_i/d_i$ are generators of the discriminant form of $\NS(W)$.
Then  there exist $t_i\in T(W)$ such that $t_i/d_i$ are generators of the discriminant form of $T(W)$ and $(n_i+t_i)/d_i$ are in $H^2(W,\Z)$. 
In practice, we can always choose $v$ primitive and set $v=t_1$ (possibly with $d_1=1$).
Then 
$\NS(W)+\langle v\rangle$ has index $d_1$ in $\NS(W_2)$, 
and the full N\'eron-Severi lattice can be obtained from $\NS(W)+\langle v\rangle$ by adjoining the vector $(n_1+t_1)/d_1\in \NS(W_2)$.

\subsection{Elliptic fibrations}
\label{ss:ell}

We will extensively use
elliptic fibrations on K3 surfaces.
To give an elliptic fibration, it suffices to exhibit a divisor $D$ of Kodaira type, 
thus coinciding with one of the singular fibers.
Then any irreducible curve meeting $D$ transversally in exactly one point gives a section of the fibration.
Elliptic fibrations with section are often called jacobian;
they can be completely understood in terms of Mordell-Weil lattices \cite{ShMW}.
A jacobian elliptic fibration on a K3 surface $X$
is determined by a direct summand of the hyperbolic plane $U$ in the 
N\'eron-Severi group:
\[
\NS(X) = U + L.
\]
Such decompositions can be generally classified 
by a gluing method going back to Kneser \cite{Kne}
and introduced to the K3 context by Nishiyama \cite{Nishi}.
A singular fiber gives rise to a negative-definite root lattice of ADE-type 
by omitting the identity component (met by the zero section $O$)
and drawing the intersection graph of the remaining components 
(which are all $(-2)$-curves and thus yield roots).
The trivial lattice of the fibration, generated by zero section and fiber components,
thus takes the shape $U+$ (root lattices of ADE-type).
Conversely the singular fibers are encoded in the roots of $L$, i.e.~in the root lattice $L_\text{root}$,
and the remainder of $\NS(X)$ comes from sections.
In detail, let $L_\text{root}'$ denote the primitive closure of $L_\text{root}$ in $L$.
Then the torsion in the Mordell-Weil group is given by
\[
\MW(X)_\text{tor} \cong L_\text{root}'/L_\text{root},
\]
and the Mordell-Weil lattice is $\MWL(X) = L/L_\text{root}'$.
Here the orthogonal projection with respect to the trivial lattice in $\NS(X)\otimes \Q$
endows $\MWL(X)$  with the structure of a definite lattice (not necessarily integral) by \cite{ShMW}.

Every jacobian elliptic fibration $X\ra C$ 
corresponds to its generic fiber,
an elliptic curve defined over the field of functions of $C$.
In consequence $X$ inherits certain automorphisms from its generic fiber. 
Every elliptic curve admits a hyperelliptic involution. 
Fiberwise this extends to the elliptic surface.
We denote the resulting involution by $-id$.  
Moreover on an elliptic curve there is translation by a point: 
Let $P$ be a rational point on an elliptic curve, 
than the induced automorphism $t_P$ send a point $Q$ to the point $Q+P$, 
where $+$ is the sum with respect to group law of the elliptic curve.
Sections, i.e.~points on the generic fiber thus induce automorphisms on the elliptic surface.

Another fundamental construction used in the following is quadratic twisting,
often also related to quadratic base change and the deck transformations. 
For background the reader is referred to \cite[Section 3.3]{HS}.
 

\subsection{Picard number \& Shioda-Inose structures}
\label{ss:SI}

In general the Picard number is far from a birational invariant, since one can always consider blow-ups.
In contrary,
for complex K3 surfaces (which are by definition minimal) the Picard number is much more than that:
by \cite{Inose} it is preserved by rational dominant maps because the Hodge structure on the transcendental lattice is preserved.
This is the main reason why
K3 surfaces of high Picard number (at least $17$) can often be studied through Kummer surfaces.
Thus the K3 surfaces in Theorems \ref{thm} and \ref{thm2}
share the structure of Kummer surfaces to a strong extent.

The most prominent case of this situation consists of a Shioda-Inose structure:
After \cite{Mo}
we ask that the K3 surface $X$ admits a rational map of degree two to a Kummer surface $\Km(A)$
such that the intersection form on the transcendental lattice is multiplied by two:
\begin{eqnarray}
\label{eq:SI}
T(\Km(A)) = T(X)(2).
\end{eqnarray}
Equivalently one has $T(A) = T(X)$.
In particular, $\Km(A)$ is the quotient of $X$ by a Nikulin involution $\jmath$.
An equivalent criterion for \eqref{eq:SI} is that $\jmath$ exchanges two perpendicular divisors of type $E_8(-1)$ (see \cite{Mo}).
Such involutions are called
Morrison--Nikulin involutions.

\subsection{Enriques involution}

Recall that an Enriques involution is a fixed point free involution $\tau$ on a K3 surface $X$.
The quotient $X/\tau$ is called Enriques surface.
Conversely we recover $X$ from $Y$ through the universal cover
\[
\pi: X\to Y.
\]
The universal cover is directly related to the canonical divisor $K_Y$ 
which gives the two-torsion in $\NS(Y)$.
Pulling back $\Num(Y)=\NS(X)/\langle K_Y\rangle \cong U + E_8(-1)$ via $\pi^*$,
we obtain a primitive embedding
\begin{eqnarray}
\label{eq:lp}
U(2) + E_8(-2) \hookrightarrow \NS(X).
\end{eqnarray}
By the Torelli theorem, Enriques involutions can be characterised by the lattice polarisation \eqref{eq:lp}
together with the additional assumption that the orthogonal complement of $U(2)+E_8(-2)$ in $\NS(X)$ does not contain any roots. 
(This makes sure that the involution determined by \eqref{eq:lp} has no fixed points.)
Thus we can study the moduli of K3 surfaces with Enriques involution through the lattice-polarisation \eqref{eq:lp}.

For instance,
any Kummer surface admits an Enriques involution by \cite{Keum}.
Contrary to this, Shioda-Inose structures do generally not accommodate Enriques involutions.
%
%
%
%

\subsection{Enriques surfaces and Brauer groups}
\label{ss: Beauville's problem}

The Brauer group of a smooth projective surface can be defined 
in \'etale cohomology as $\Br(S) = H_\text{\'et}^2(S,\mathbb G_m)$.
The Brauer group is a birational invariant that encodes very subtle information.
For instance if $S$ is a complex surface,
then $\Br(S)$ contains a subgroup $\Br(S)'$ 
which is dual to the  transcendental lattice in a suitable sense,
and the quotient $\Br(S)/\Br(S)'$ is isomorphic to the torsion in $H^3(S,\Z)$.

For a complex Enriques surface $Y$ it follows that
\[
\Br(Y)=H_\text{\'et}^2(Y, \mathbb G_m)=
\Z/2\Z.
\]
Consider the K3 surface $X$ given by the universal cover $\pi: X \to Y$.
The important
problem how $\Br(Y)$ pulls back to the K3 surface $X$ via $\pi^*$
was recently solved by Beauville:

\begin{Theorem}[Beauville {\cite{Beau}}]
\label{thm:Beau}
Generally $\pi^*\Br(Y)=\Z/2\Z$ holds.
One has $\pi^*\Br(Y)=\{0\}$ if and only if there is a divisor $D$ on $X$ such that $\tau^*D=-D$ in $\NS(X)$ and $D^2\equiv 2\mod 4$.
\end{Theorem}

In other words, the Enriques surfaces with Brauer group pulling back identically to zero to the covering K3 surface lie on countably many hyperplanes in the moduli space of Enriques surfaces (cut out by the conditions of \eqref{eq:prim} below).

\textbf{Problem (Beauville):} 
{\sl
\begin{enumerate}
\item
Give explicit examples of Enriques surfaces such that $\pi^*(\Br(Y))=\{0\}$. 
\item
Are there such surfaces defined over $\Q$?
\item
If so, exhibit some.
\end{enumerate}}

Note that the main problem in (2) consists in the possibility that the countable number of Enriques surfaces in question might avoid the specific hyperplanes.

In the meantime, we have seen isolated examples as singular K3 surfaces answering all three questions (cf.~\cite{GvG}, \cite{HS}),
but no explicit families yet.
Let us emphasise that here we ask for explicit defining equations as opposed to (moduli spaces of) K3 surfaces determined by a lattice polarisation.
We will make this difference clear below.
Single examples over $\Q$ have been exhibited independently by one of us with B.~van Geemen \cite{GvG} (see Section \ref{ss:GvG}) and by the other with K.~Hulek in \cite{HS}.
Note that the first example is a Kummer surface while the second is not.

Our aim is to exhibit explicit families of K3 surfaces with Enriques involution as above.
Abstractly this is easily achieved lattice-theoretically as we only require the K3 surface $X$ to admit a primitive embedding
\begin{eqnarray}
\label{eq:prim}
U(2) + E_8(-2) + \langle-2N\rangle \hookrightarrow\NS(X)
\end{eqnarray}
for some odd $N>1$ such that the orthogonal complement of $U(2)+E_8(-2)$ in $\NS(X)$ 
does not contain any $(-2)$ vectors.
However, it is non-trivial to exhibit explicit geometric constructions of such surfaces,
let alone find explicit equations.

\section{Lattice Enhancements}
\label{s:lat}
In this section we will construct the surfaces $X_N$, 
appearing in Theorem \ref{thm} using  lattice theory 
and in particular the construction described in the Section \ref{construction: specialisation}. 
The surfaces $X_N$ will be the generic members of the family $\mX_N$ 
of the $\left( U(2)+2E_8(-1)+\langle-2N\rangle\right)$-polarised K3 surfaces. 
As this lattice admits a unique embedding into $\Lambda_{K3}$ 
up to isometries (cf. Theorem \ref{theorem: Nikulin 1.14.4}), 
K3 surfaces with this polarisation form a unique one-dimensional family.
It will be convenient to view $\mX_N$ as subfamilies of 
the Barth--Peters family $\mX$.
This is a 2-dimensional family of K3 surfaces 
which admits an Enriques involution with exceptional properties \cite{BP} (see also \cite{MN}, \cite{HS}). 
The Barth--Peters family $\mX$ specializes to the 1-dimensional families $\mX_N$ (cf. Section \ref{ss:X_N}).

\subsection{Barth--Peters family $\mX$}\label{ss: BP family 1}  
%
There is a unique two-dimensional family of K3 surfaces $\mX$ such that generally
\begin{eqnarray}
\label{eq:NS-X}
\NS(\mX)=U(2)+2E_8(-1).
\end{eqnarray}
Here the primitive embedding \eqref{eq:prim} is achieved  
by realising $E_8(-2)$ diagonally in the two copies of $E_8(-1)$.
By construction, this induces an Enriques involution $\tau$ on the general member $X$ of $\mX$.

There are many ways to exhibit $\mX$ geometrically, see \cite{BP}, \cite{MN}.
For instance, one can give it as two-dimensional family of elliptic fibrations
\begin{eqnarray}\label{formula:Weierstrass form S_2} 
\;\;\; \;\;\; \mX:\;\;\; y^2=x(x^2+a(t)x+1),\ \ \
a(t)=a_0+a_2t^2+t^4\in k[t].
\end{eqnarray}
There is a 2-torsion section $(0,0)$ and a reducible fiber of Kodaira type $I_{16}$ at $\infty$.
The above fibrations are quadratic base changes from a one-dimensional family $\mathcal R$
of rational elliptic surfaces
that can be recovered as quotient by the involution $\imath$ induced by $t\mapsto -t$.
The composition of $\imath$ and translation by $(0,0)$ is an Enriques involution
(the classical case of the more general construction from \cite{HS}).

%
%

%
In order to exhibit a basis of $\NS(\mX)$,
we note that the rational elliptic surfaces in $\mathcal R$ generally have Mordell-Weil rank one.
Pulling back a generator,
we obtain a section $Q$ 
on $\mX$ of height $h(Q)=1$.
Comparing discriminants
we find that the Mordell-Weil lattice of this elliptic fibration is generated by $Q$: $\MWL(\mX)\simeq [1]$.

\subsection{The subfamilies $\mX_N$ of $\mX$}
\label{ss:subfam}
\label{ss:X_N}
Starting from the Barth--Peters family $\mX$, we want to describe the K3 surfaces with Picard number 19
and N\'eron-Severi lattice isometric to $U(2)+2E_8(-1)+\langle -2N\rangle$. 
Under a very mild condition, the Enriques involution specialises also from $\mX$ as we will see in \ref{ss:N}. 

\begin{prop}
\label{prop: sp X to XN} 
Let $X_{N}$ be the generic member of the subfamily of the Barth--Peters family $\mX$ 
obtained as in  \ref{construction: specialisation} by
choosing the vector $v$ to be $v_N:=(1,-N,0,0)\in U+ U(2)\simeq T(X)$. 
Then $\NS(X_{N})\simeq \NS(X)+ \langle -2N\rangle$, $T(X_N)\simeq \langle 2N\rangle + U(2)$. 
\begin{enumerate}
\item
If $N>1$, then $X_{N}$ admits an elliptic fibration with singular fibers $I_{16}+8I_1$, 
Mordell-Weil group $(\Z)^2\times \Z/2\Z$ and Mordell--Weil lattice $MWL\simeq \left[\begin{array}{ll}1&0\\0&2N\end{array}\right]$. 
\item
If $N=1$, then the induced elliptic fibration on $X_1$ has singular fibers $I_{16}+I_2+6I_1$ 
and the Mordell--Weil group is $\Z\times \Z/2\Z$. 
\end{enumerate}
\end{prop}

\proof The transcendental lattice of $X_{N}$ is the orthogonal complement in $T(X)$ to $v_N$, 
so it is isometric to $\langle 2N\rangle+ U(2)$. 
In particular the transcendental lattice has discriminant $2^3N$. 
The N\'eron--Severi is an overlattice of $\NS(X)+ \langle -2N\rangle$ of finite index. 
Since the discriminant of $\NS(X)+ \langle -2N\rangle$ is $-2^3N$, 
we conclude that $\NS(X_{N})\simeq \NS(X)+ \langle -2N\rangle$. We denote by $F$ the class of the fiber of the elliptic fibration  \eqref{formula:Weierstrass form S_2} on $X$ and by $O$ the class of the zero section.
The elliptic fibration on $X$ specializes to an elliptic fibration on $X_N$. 
If $N>1$, then the class $u:=NF+O+v_N$ corresponds to a section of infinite order on $X_N$. 
The section $u$ meets the reducible fiber in the identity component $C_0$ of the $I_{16}$ fiber, 
$u\cdot O=N-2$, $u\cdot Q=N$. 
This gives the Mordell--Weil lattice.

If $N=1$ the class $v_N$ corresponds to a class with self--intersection $-2$ 
which is orthogonal to the class of the fiber and to the fiber components of $I_{16}$.
Hence on the fibration there is another reducible fibers, which is of type $I_2$.  
So the elliptic fibration on $X_1$
has $I_{16}+I_2+6I_1$ as singular fibers  and the Mordell--Weil group is $\Z\times \Z/2\Z$ as before. 
\eprf


We will work out an explicit geometric construction of $X_N$ for odd $N$ in Section \ref{s:geom}.
Meanwhile this section is concluded with an investigation how the Enriques involution $\tau$ on $X$ specialises to $X_N$.

\subsection{Enriques involution on $X_N$}
\label{ss:N}

On K3 surfaces, specialisation preserves many properties such as automorphisms.
Along these lines, an Enriques involution will specialise to an involution, but it need not specialise to an Enriques involution.
That is, the specialised involution need not be fixed point free anymore.
This subtlety is based on the fact that the moduli space of Enriques surface is exactly the moduli space of $\left(U(2)+E_8(-2)\right)$--polarised K3 surfaces with countably many hyperplanes removed.
Whence one has to avoid the situation where the specialisation hits (or even sits inside) those hyperplanes.

The hyperplanes correspond to the presence of some $(-2)$ curve in the orthogonal complement of $U(2)+E_8(-2)$ inside $\NS$.
In particular, we have seen an instance of an Enriques involution not specialising in \ref{ss:subfam}:
For $N=1$, the singular fibers degenerate on $X_N$ to form an $I_2$ fiber 
(where the corresponding base change ramifies).
Naturally this gives a $(-2)$ curve in the specified orthogonal complement, 
so the family of $X_1$ lies completely in one such hyperplane.
We will now check that this does not happen for $N>1$:

\begin{prop}
The Enriques involution $\tau$ on the Barth--Peters family $\mX$
specialises to an Enriques involution on the subfamily $\mX_N$ if and only if $N>1$.
\end{prop}

\begin{proof}
We start with the primitive embedding of $U(2)+ E_8(-2)$  in $\NS(X)$ given by the Enriques involution $\tau$ on $\mX$.
Clearly this induces a primitive embedding
\[
U(2)+ E_8(-2)\hookrightarrow\NS(X_N)\simeq \NS(X)+ \langle -2N\rangle.
\]
The orthogonal complement of $U(2)+ E_8(-2)$ in $\NS(X_N)$ is thus isometric to 
\begin{eqnarray}
\label{eq:orto}
(U(2)+ E_8(-2))^{\perp_{\NS(X)}}+ \langle -2N\rangle\simeq E_8(-2)+ \langle -2N\rangle.
\end{eqnarray}
Note that  the orthogonal complement of $U(2)+ E_8(-2)$ in $\NS(X_N)$ is negative-definite.
It contains no classes with self--intersection $-2$ if and only if $N>1$.
Hence it is exactly the latter case where $\tau$ specialises to an Enriques involution on $X_N$. 
\end{proof}

\subsection{Abstract proof of Theorem \ref{thm}}

Thanks to the specific form of our K3 surfaces and the Enriques involution,
we can determine explicitly how the Brauer group pulls back from the  Enriques quotient.
This enables us to prove Theorem \ref{thm}. 

Recall the setup with
$N>1$
and $X_N$ a general member of the K3 family $\mX_N$.
Let $\tau$ denote the  Enriques involution induced from $\mX$
and $Z_N=X_N/\tau$.

\label{prop:br}

We have computed the orthogonal complement of $U(2)+E_8(-2)$ in $\NS(X_N)$ in \eqref{eq:orto}.
Clearly this gives exactly those divisors which are anti-invariant for $\tau^*$.
The lattice in \eqref{eq:orto} represents only four-divisible integers if and only if $N$ is even.
Thus we deduce from Theorem \ref{thm:Beau} that
\[
\tau^* \Br(Z_N) =
\begin{cases}
\{0\} & \text{ if $N$ is odd},\\
\Z/2\Z & \text{ if $N$ is even}.
\end{cases}
\]
This proves Theorem \ref{thm}. \qed


\begin{rem}
The same argument applies to the Barth--Peters family $\mX$ to show that for a general member the second alternative (injectivity of the Brauer group under pull-back) holds true.
\end{rem}



\section{Barth--Peters family: elliptic fibrations and moduli}\label{s: BP elliptic fibr and moduli}

In this section and in the next one 
we will describe geometric properties and elliptic fibrations of the families 
introduced in Section \ref{s:lat} in order to describe their moduli spaces 
and to exhibit a geometric proof  of Theorem \ref{thm}. 
In particular we will associate to the Barth--Peters family $\mX$
a family of Kummer surfaces and hence of abelian surfaces. 
Using the relations between these families one can easily describe the moduli 
and the Picard--Fuchs equation of the Barth--Peters family.
This answers a problem on Enriques Calabi-Yau threefolds 
originating from string theory.

\subsection{The elliptic fibration $[2III^*,2I_2]$ on $\mX$}\label{ss: BP elliptic 2III^* 2I2}
We choose another convenient model of the Barth-Peters family $\mX$ of K3 surfaces following \cite{HS}, \cite{MN}.
It is defined as jacobian elliptic fibration through a family of quadratic base changes of $\PP^1$
\[
f: t\mapsto \frac{(t-a)(t-b)}t,\;\;\; ab\neq 0
\]
over the unique rational elliptic surface $R$ with singular fibers $III^*, I_2, I_1$ and $\MW=\Z/2\Z$.
One finds the  model
\begin{eqnarray}
\label{eq:BP}
\mX:\;\;\; y^2 = x(x^2 + t^2 x + t^3(t-a)(t-b))
\end{eqnarray}
generally with reducible fibers of type $III^*$ at $0, \infty$ and $I_2$ at $a,b$.
Here the two-torsion section is given by $(0,0)$.
Despite the symmetry in $a, b$,
it is natural to study the family in the parameters $a,b$ (as opposed to $a+b, ab$),
since we want to parametrise K3 surfaces with $\rho=19$ over a given field (say over $\Q$),
i.e.~without Galois action on the two $I_2$ fibers.

\subsection{Enriques involution on $\mX$}\label{ss:cti}
On $\mX$ we have several interesting involutions.
We will need the following:
\begin{itemize}
\item
the deck transformation corresponding to $f$;
\item
translation by the two-torsion section $(0,0)$;
\item
the hyperelliptic involution $-id$.
\end{itemize}
As in \cite{HS},
the composition of the first two involutions defines an Enriques involution $\tau$ on a general member of the family $\mX$.
It was checked in \cite{HS} that this is exactly the involution induced by the decomposition \eqref{eq:NS-X} of $\NS(\mX)$ and the specified embedding of the Enriques lattice.
Denote the quotient family by $\mY=\mX/\tau$.
Then the hyperelliptic involution $-id$ induces an involution on $\mY$
which acts trivially on $H^2(\mY,\Z)$.
Such a  cohomologically trivial involution is  remarkable since it cannot occur on a K3 surface by the Torelli theorem.
In fact, complex Enriques surfaces with cohomologically trivial involution have been classified by Mukai and Namikawa \cite{MN} (later corrected by Mukai \cite{Mukai}):

\begin{Theorem}
Let $Y$ be a complex Enriques surface with a cohomologically trivial involution.
Then $Y\in\mY$.
\end{Theorem}

\subsection{Relation with Kummer surfaces}
\label{ss:222}
The Barth-Peters family admits a Shioda-Inose structure (cf. Section \ref{s:Km}), but it will be even more convenient for our purposes to pursue a different approach leading to Kummer surfaces.
Namely  we will study the family $\mX$ by applying a suitable symplectic involution 
such that the quotient family consists of Kummer surfaces of product type.

In order to relate the family $\mX$ directly to some Kummer surfaces,
we consider an alternative elliptic fibration.
We proceed by identifying suitable divisors of Kodaira type (cf.~\ref{ss:ell}).
Presently we extract two singular fibers of type $I_4^*$ from the curves visible in the  elliptic fibration \eqref{eq:BP}.
One divisor of Kodaira type $I_4^*$ is supported on the $III^*$ fiber at $0$ extended by zero section and identity components of $III^*$ at $\infty$ and one $I_2$ fiber (say at $t=a$).
The perpendicular curves (components of the $III^*$ at $t=\infty$ plus two-torsion section, far simple component of $III^*$ at $t=0$ and of $I_2$ at $t=b$) form another divisor of type $I_4^*$.
This leaves two double components of the $III^*$ fibers that serve as sections of the new fibration (zero and two-torsion).
All these $(-2)$-curves are sketched in the following figure:

\begin{figure}[ht!]
\setlength{\unitlength}{.35in}
\begin{picture}(10,4.2)(0,0)

\thinlines

\multiput(0,0)(1,0){5}{\circle*{.1}}
\multiput(0,4)(1,0){5}{\circle*{.1}}
\multiput(0,1)(0,1){3}{\circle*{.1}}
\multiput(4,1)(0,1){3}{\circle*{.1}}
\multiput(2,1)(0,2){2}{\circle*{.1}}
\multiput(1,2)(2,0){2}{\circle*{.1}}

\put(0,0){\line(1,0){4}}
\put(0,4){\line(1,0){4}}
\put(0,0){\line(0,1){4}}
\put(4,0){\line(0,1){4}}

\multiput(0,2)(3,0){2}{\line(1,0){1}}
\multiput(2,0)(0,3){2}{\line(0,1){1}}



\put(-.25,3.5){\line(1,-1){3.75}}
\put(-.25,3.5){\line(0,-1){3.75}}
\put(-.25,-.25){\line(1,0){3.75}}

\put(.5,4.25){\line(1,-1){3.75}}
\put(.5,4.25){\line(1,0){3.75}}
\put(4.25,4.25){\line(0,-1){3.75}}

\put(0,4){\circle{.2}}
\put(4,0){\circle{.2}}

\put(4.8,1.9){$\leftrightsquigarrow$}

\multiput(6,0)(1,0){5}{\circle*{.1}}
\multiput(6,4)(1,0){5}{\circle*{.1}}
\multiput(6,1)(0,1){3}{\circle*{.1}}
\multiput(10,1)(0,1){3}{\circle*{.1}}
\multiput(8,1)(0,2){2}{\circle*{.1}}
\multiput(7,2)(2,0){2}{\circle*{.1}}

\put(6,0){\line(1,0){4}}
\put(6,4){\line(1,0){4}}
\put(6,0){\line(0,1){4}}
\put(10,0){\line(0,1){4}}

\multiput(6,2)(3,0){2}{\line(1,0){1}}
\multiput(8,0)(0,3){2}{\line(0,1){1}}

\put(5.75,-.25){\framebox(1.5,4.5){}}

\put(8.75,-.25){\framebox(1.5,4.5){}}

\put(8,4){\circle{.2}}
\put(8,0){\circle{.2}}

\end{picture}
\caption{Divisors of type $I_4^*$ vs.~$III^*$'s and $A_1$'s}
\label{Fig:BP-basic}
\end{figure}
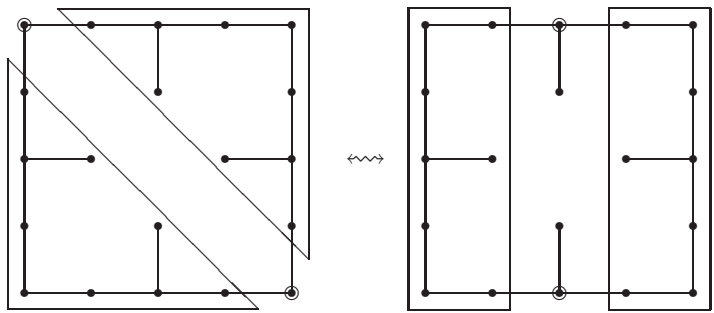

Explicitly, this elliptic fibration is extracted by the parameter $u=x/(t^2(t-a))$ in \eqref{eq:BP}.
We obtain the Weierstrass form (in $t,y$)
\begin{eqnarray}
\label{eq:4*}
\mX:\;\;\; y^2 = t (t^2+u(1+u-au^2)t-b u^4).
\end{eqnarray}
Here translation by the two-torsion section $(0,0)$ defines a Nikulin involution.
After desingularisation, the quotient results in a family $\mX'$ of K3 surfaces with two singular fibers of type $I_2^*$ and four fibers of type $I_2$:
\begin{eqnarray}
\label{eq:X'}
\mX':\;\;\; y^2 = t (t^2-2u(1+u-au^2)t+u^2((1+u-au^2)^2+4b u^2)).
\end{eqnarray}
These elliptic surfaces have generally $\MW=(\Z/2\Z)^2$ over $k(\sqrt{-b})$ (given explicitly below).

\subsection{Kummer structure}
\label{ss:Km-str}

By the classification of Oguiso \cite{Oguiso}, a general Kummer surface of product type $\Km(E\times E')$
admits an elliptic fibration with singular fibers and  $\MW$ as above.
Thus we compare two-dimensional families of K3 surfaces:
$\mX'$ and Kummer surfaces of product type.
But here it follows from the discriminant form by Proposition \ref{prop: uniquely of a lattice in primitve lattice, condition on discriminant} (or from Oguiso's classification) that 
\[
\NS = U+2D_4(-1)+E_8(-1).
\]
As this lattice admits a unique embedding into the K3 lattice up to isometries (cf. Theorem \ref{theorem: Nikulin 1.14.4}),
K3 surfaces with this N\'eron-Severi lattice form a unique two-dimensional family.
In particular, the family $\mX'$ and the Kummer family of product type coincide.
We proceed by working out the relation in detail.

Given $\mX'$ over $k(a,b)$, there are elliptic curves $E, E'$ such that $\mX'\cong \Km(E\times E')$.
In order to find the  elliptic curves,
we exhibit an alternative elliptic fibration on $\mX'$ with two fibers of type $IV^*$.
This will allow us to obtain information about the j-invariants of the elliptic curves from
the coefficients of the Weierstrass form by \cite{Inose} (cf.~\cite{Sandwich}).

We identify two disjoint divisors of Kodaira type $IV^*$ in the model \eqref{eq:X'} as depicted in Figure \ref{Fig:BP-2}:
on the one hand, the first five components of an $I_2^*$ fiber (say at $\infty$) extended by zero section $O$ and the two-torsion section $R=(0,0)$ (which is distinguished by the fact that it meets all reducible fibers at non-identity components);
on the other hand, the last five components of the other $I_2^*$ fiber extended by the other two-torsion sections.
These divisors induce an elliptic fibration on $\mX'$ with two singular fibers of type $IV^*$.
Here we have plenty of sections for the new fibration given by the remaining original fiber components.


\begin{figure}[ht!]
\setlength{\unitlength}{.35in}
\begin{picture}(11,2)(0,-0.2)

\thinlines

\multiput(4,0)(2,0){2}{\circle*{.1}}
\multiput(9.5,0)(1,0){1}{\circle*{.1}}
\multiput(5,0)(5.5,0){2}{\circle{.1}}
\put(.5,0){\circle*{.1}}
\put(0,0){\line(1,0){.5}}
\put(4,0){\line(1,0){2}}
\put(9.5,0){\line(1,0){1.5}}

\multiput(4,1.5)(2,0){2}{\circle*{.1}}
\multiput(9.5,1.5)(1,0){1}{\circle*{.1}}
\multiput(5,1.5)(5.5,0){2}{\circle{.1}}
\put(.5,1.5){\circle*{.1}}
\put(0,1.5){\line(1,0){.5}}
\put(4,1.5){\line(1,0){2}}
\put(9.5,1.5){\line(1,0){1.5}}

\multiput(1.25,.75)(1,0){3}{\circle*{.1}}
\multiput(6.75,.75)(1,0){3}{\circle*{.1}}
\multiput(.5,0)(5.5,0){2}{\line(1,1){.75}}
\multiput(.5,1.5)(5.5,0){2}{\line(1,-1){.75}}
\multiput(4,0)(5.5,0){2}{\line(-1,1){.75}}
\multiput(4,1.5)(5.5,0){2}{\line(-1,-1){.75}}
\multiput(1.25,.75)(5.5,0){2}{\line(1,0){2}}

\put(4.55,1.08){$O$}
\put(4.55,.15){$R$}

\multiput(.5,0)(0,1.5){2}{\circle{.2}}
\multiput(6,0)(0,1.5){2}{\circle{.2}}

\put(1,-.25){\framebox(4.25,2){}}
\put(1.1,1.3){$IV^*$}

\put(6.5,-.25){\framebox(4.25,2){}}
\put(6.6,1.3){$IV^*$}

\end{picture}
\caption{Divisors of type $IV^*$ vs.~$I_2^*$'s and two-torsion sections}
\label{Fig:BP-2}
\end{figure}
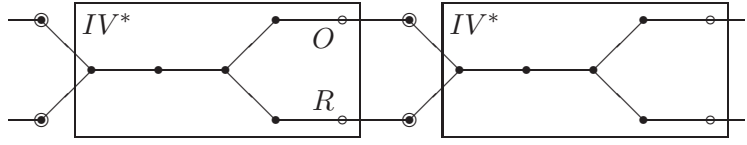

To write down the fibration explicitly, it is convenient to translate $x$ so that one of the other two-torsion sections becomes $(0,0)$.
For this purpose, we write $b=-c^2$.
Then the conjugate two-torsion sections have $t$-coordinate
$u+u^2\pm 2 c u^2-a u^3$.
The translation $t\mapsto t+u+u^2+ 2 c u^2-a u^3$ gives the Weierstrass form
\[
\mX':\;\;\; y^2 = t (t+4 c u^2) (t+u+u^2+2 c u^2-a u^3).
\]
The last factor of the RHS encodes the distinguished two-torsion section $R$.
The above divisors of Kodaira type $IV^*$ are extracted at $v=0,\infty$ by the affine parameter
\[
v = \dfrac{y}{t+u+u^2+2 c u^2-a u^3}.
\]
Indeed, solving for $y$, we obtain the following family of cubics in $\A^2$ with coordinates $t,u$
and parameter $v$:
\[
\mX':\;\;\; (t+u+u^2+2 c u^2-a u^3) v^2 = t (t+4 c u^2).
\]
This model makes visible the quadratic base change from the rational elliptic surface that is given
by the cubic pencil with $w=v^2$.
Standard formulas give the following Weierstrass form in the usual coordinates $x,y$ with elliptic parameter $v$ and moduli $a, b$ recovered from $b=-c^2$:
\[
\mX':\;\;\; y^2 = x^3
-\frac{16}3 v^4 (1-12 b+3 a) x
+\frac{16}{27} v^4 (8 v^2+288 v^2 b+36 a v^2-432 b+27 a^2 v^4).
\]

\subsection{Elliptic curves}
\label{ss:ec}

Recall that the two families coincide,
i.e.~there are elliptic curves $E, E'$ such that $\mX'=\Km(E\times E')$.
Here we can compute the j-invariants $j, j'$ as follows.
The variable change
\[
v \mapsto 
2(-b)^{1/4}v/\sqrt{a}, 
x \mapsto 16 (b/a)^{2/3} x
\]
leads to the standard normal form
\[
y^2 = x^3 - 3Av^4x - v^4(v^4+2Bv^2+1)
\]
where $A, B$ are algebraic expressions in $a,b$.
By the work of Inose \cite{Inose} (cf.~\cite{Sandwich}), $jj'=12^6A^3$ and $(j-12^3)(j'-12^3)=12^6B^2$
(so that $A$ and $B$ are products of Weber functions).
In the present situation, one obtains
\begin{eqnarray}
\label{eq:*}
jj' = -4096 (3 a-12 b+1)^3/(a^2 b)
\end{eqnarray}
and
\begin{eqnarray}
\label{eq:+}
(j-12^3)(j'-12^3)=-1024 (9 a+72 b+2)^2/(a^2 b).
\end{eqnarray}
Thus we can express the elliptic curves $E, E'$ in terms of the moduli $a,b$.
Note that we lost the symmetry in $a,b$
when extracting the two $I_4^*$ fibers on $\mX$.
Algebraically in the above formulas, this can be accounted for as follows:
if $j$ gives a solution of the system \eqref{eq:*}, \eqref{eq:+} for the ordered pair $(a,b)$,
then $mj+l$ encodes a solution of the system for the ordered pair $(b,a)$ with
\begin{eqnarray*}
m & = & \dfrac{a (64 a^2+16 a+16 b a+b^2)}{b(64 b^2+16 b+16 b a+a^2)},\\
l & = & \dfrac{8 (156 b^2 a-4 b^3+16 a+128 a^2+80 b a-b^2+256 a^3-192 b a^2)}{b^2 a}\\
&& \;\;\; + m\dfrac{8 (156 b a^2-4 a^3+16 b+128 b^2+80 b a-a^2+256 b^3-192 b^2 a)}{a^2 b}.
\end{eqnarray*}


\subsection{Conclusion}\label{conclusion}
{\sl
The Hodge structure on $\mX$ is given by the pair $(E, E')$ as above.}

For instance, if $E, E'$ are isogenous, but without CM, then
any $X$ as above will have Picard number $\rho(X)=19$.
Note, however, that it is not clear from these computations how we can choose $a, b$ so that 
$X$ attains a chosen transcendental lattice of rank two or three.
This problem will be overcome for the cases related to Theorem \ref{thm} 
by geometric means in Section \ref{s:geom}.

\subsection{Relations with Physics and Picard--Fuchs equations}

An Enriques Calabi--Yau threefold is the smooth quotient 
$\left(S\times E\right)/\left(\tau\times \iota_E\right)$, 
where $S$ is a K3 surface admitting an Enriques involution $\tau$, 
$E$ is an elliptic curve and $\iota_E$ is the hyperelliptic involution on $E$. 
These particular threefolds are intensively studied in the context of mirror symmetry and string theory, 
also because they are their own mirror. 
In a certain sense this property depends on the corresponding property for the K3 surfaces: 
the family of K3 surfaces admitting an Enriques involution is its own mirror 
within the framework of mirror symmetry of polarised K3 surfaces.

It is immediate to check with the K\"unneth formula
that $h^{2,1}((S\times E)/(\tau\times \iota_E))=11$.
Hence the family of Enriques Calabi--Yau threefolds is  $11$-dimensional. 
Note that the dimension of the family of the Enriques Calabi--Yau threefolds 
is the sum of the dimensions of the family of the K3 surfaces 
and of the elliptic curve involved in the construction.
Thus all the deformations of the threefolds are induced 
by deformations of the K3 surface and of the elliptic curve. 

To gain specific insight into Enriques Calabi-Yau threefolds,
recently certain subfamilies have drawn considerable attention.
In particular, in \cite{KM} the Barth-Peters family $\mX$
has been studied from this view point.

In order to describe the mirror map for the resulting families of Calabi-Yau threefolds, 
the Picard-Fuchs equation of the family of K3 surfaces $\mX$ is computed in \cite[(6.26)]{KM}. 
Since the Barth-Peters family is a 2-dimensional family, 
one expects that the Picard-Fuchs equation is a partial differential equation of order 4.
However, in this particular case the Picard-Fuchs equation 
splits into a system of two partial differential equations of order 2
which  can be solved separately. 

Our construction provides a geometric interpretation of this result through Kummer structures. 
Indeed we proved (see Section \ref{conclusion}) that the variation of the Hodge structures of $\mX$ depends only on the variation of the Hodge structures of two non-isogenous elliptic curves (and the Picard--Fuchs equation of a family of elliptic curves is a second order differential equation).
In addition, one obtains the same Picard--Fuchs equations for several other families of K3 surfaces
that are related by rational dominant maps 
between the generic members (such as $\mY$ in Section \ref{s:Km}).
 Naturally this property carries over to subfamilies.
 Along these lines, one can find the Picard--Fuchs equations for the families related
 to Theorem \ref{thm} and \ref{thm2}
 (see \ref{ss:X_N}, \ref{ss:Km})
 through the results from \cite{CDLW}.

\section{Geometric construction}
\label{s:geom}

This section may be considered the geometric heart of this paper 
as we construct explicitly the K3 surfaces in Theorem \ref{thm}: 
here we exhibit in a purely geometric way the surfaces $X_N$ 
that were introduced in Section \ref{s:lat} from the point of view of the lattices.

\subsection{Outline}

Given $N$-isogenous elliptic curves $E, E'$,
we consider a particular elliptic fibration on the Kummer surface $\Km(E\times E')$.
If $N>1$, then the graph of the isogeny induces an additional section.
This section can be traced through two related elliptic fibrations 
until we reach the fibration from \ref{ss:222}.
Then the quotient by a two-torsion section takes us to a member $W_N$ of the Barth--Peters family $\mX$.
In fact, there are two ways to go through this whole procedure.
In each case, we compute the transcendental lattice $T(W_N)$,
and one case leads to Theorem \ref{thm} (cf.~\ref{ss:2nd}).

\subsection{Abelian surface}

Let $E, E'$ denote complex elliptic curves without CM.
Assume that they are (cyclically) $N$-isogenous.
Then the abelian surface $A=E\times E'$ has transcendental lattice $T(A)=U+\langle 2N\rangle$.

\subsection{Kummer surface}
\label{ss:Km}

It follows that the Kummer surface $\Km=\Km(E\times E')$
has transcendental lattice $T(\Km)=U(2)+\langle 4N\rangle$.
We consider three specific elliptic fibrations 
that also live on general Kummer surfaces of product type,
i.e.~Kummer surfaces of a product of non-isogenous elliptic curves 
(as classified by Oguiso \cite{Oguiso}).
We write the fibrations in terms of the reducible fibers and torsion in $\MW$ in the non-degenerate case $N>1$ :
\begin{enumerate}
\item
$[II^*, 2I_0^*]$, $\MW_\text{tor}=\{0\}$;
\item 
$[III^*, I_2^*, 3I_2]$, $\MW_\text{tor}=\Z/2\Z$;
\item
$[2I_2^*, 4I_2]$, $\MW_\text{tor}=(\Z/2\Z)^2$.
\end{enumerate}
If $E, E'$ were not isogenous, then the above fibrations would have $\MW$-rank zero with $\NS$ fully generated by the given singular fibers and sections.
The generic fibration of the third kind has already appeared in \ref{ss:222}.
It is also instructive to note that while the first fibration is unique on $\Km$ up to $\Aut(\Km)$,
the second and third are generally not.
In fact, by \cite{Oguiso} there are generally six resp.~nine inequivalent such fibrations.
This property will only enter implicitly in our construction (cf.~Remark \ref{rem:field} and \ref{ss:2nd}).

\subsection{First elliptic fibration}

This fibration has played a central role in the study of singular K3 surfaces (cf.~\cite{Sandwich}, \cite{SI}).
Like the fibration with two $IV^*$ fibers in \ref{ss:ec}, 
the coefficients of the Weierstrass form admit a simple algebraic expression in the j-invariants of the elliptic curves.
As a further advantage, it is easy to determine the abstract shape of a section induced by an isogeny of the elliptic curves.

The $N$-isogeny between $E$ and $E'$ induces an additional divisor on the above fibration.
If $N=1$, then this is a fiber component, as one of the $I_0^*$ fibers degenerates to type $I_1^*$.
In the following, we only consider the case $N>1$.
Then the additional divisor can be represented by a section $P$ on the above elliptic fibration.
We employ the theory of Mordell-Weil lattices \cite{ShMW} to find information about the section $P$.

\begin{Lemma}
The section $P$ meets either one or both fibers of type $I_0^*$ at a non-identity component depending on the parity of $N$ being odd or even.
\end{Lemma}

\begin{proof}
Recall the trivial lattice $U+2D_4(-1)+E_8(-1)$ generated by zero section and fiber components.
Since the trivial lattice has discriminant $-16$ while $\Km$ has discriminant $16N$, the section $P$ ought to have height $N$.
We will use that $P$ cannot meet both $I_0^*$ fibers at their identity components.
Otherwise, it would be orthogonal to the two copies of $D_4(-1)$ from the trivial lattice inside $\NS$,
and thus the $2$-length of $\NS$ would be at least four, exceeding the rank of $T(\Km)$ which is three.
Hence the height of $P$ is
\[
h(P) = 4 + 2 (P.O) - \begin{cases}
1 &  \text{ if $P$ meets one } D_4,\\
2 &  \text{ if $P$ meets both } D_4\text{'s}.
\end{cases}
\]
Since $h(P)=N$, the intersection behaviour is predicted by the parity of $N$ as claimed.
\end{proof}

From the lemma and the height of $P$, we also obtain the intersection number
\[
(P.O) = 
\begin{cases}
(N-3)/2 &  \text{ if $N$ is odd},\\
(N-2)/2 &  \text{ if $N$ is even}.
\end{cases}
\]

\subsection{Second elliptic fibration}
\label{ss:fib2}

From the first elliptic fibration, we will extract an elliptic fibration of the second kind as in \ref{ss:Km}.
From here on, we concentrate on the case where $N>1$ is odd.

We identify a divisor of Kodaira type $I_2^*$ on $\Km$ as follows:
Take the $I_0^*$ fiber met by $P$ in a non-identity component minus exactly that component and extend by zero section and identity components of the other two reducible fibers.
This induces an elliptic fibration on $\Km$ with the second components of $II^*$ and the other $I_0^*$ as sections (zero and two-torsion).
Perpendicular to the new $I_2^*$ fiber we find an $E_7$ coming from $II^*$ and three $A_1$'s coming from $I_0^*$.
Thus we obtain exactly the second elliptic fibration from \ref{ss:Km}.

In Figure \ref{Fig:fib2} we sketch these divisor classes.
Depicted the reader finds the zero section $O$ of the first fibration
and the components of the reducible fibers.
We also include the section $P$ of the first fibration
that intersects $O$ with multiplicity $(N-3)/2$.
For the second fibration,
we mark the components of the two ``big'' singular fibers by boxes
and the new sections $O', R'$ by circles.

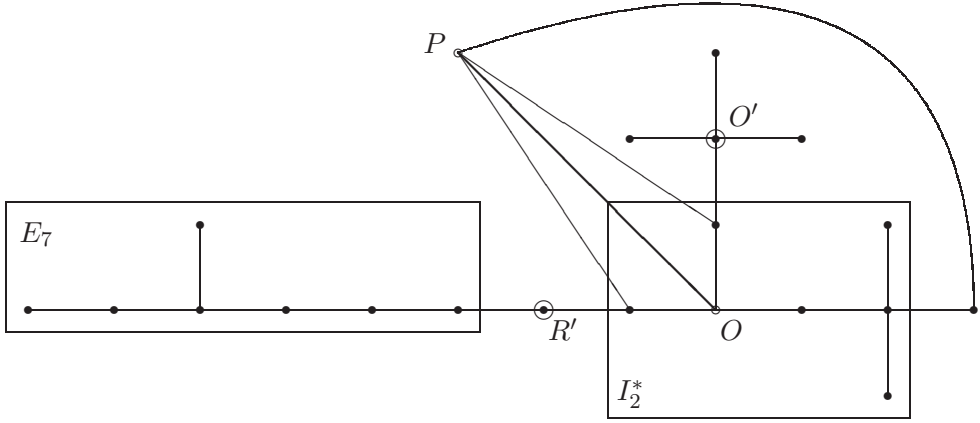
\begin{figure}[ht!]
\setlength{\unitlength}{.45in}
\begin{picture}(11,4.8)(0,-0.25)
\thinlines

\put(8,1){\circle{.1}}
\put(8.05,.65){$O$}

\put(6.05,.65){$R'$}

\multiput(0,1)(1,0){8}{\circle*{.1}}

\multiput(9,1)(1,0){3}{\circle*{.1}}

\put(0,1){\line(1,0){11}}

\multiput(7,3)(1,0){3}{\circle*{.1}}
\put(7,3){\line(1,0){2}}

\multiput(8,2)(0,2){2}{\circle*{.1}}
\put(8,1){\line(0,1){3}}

\multiput(10,0)(0,2){2}{\circle*{.1}}
\put(10,0){\line(0,1){2}}

\put(2,1){\line(0,1){1}}
\put(2,2){\circle*{.1}}

\put(6,1){\circle{.2}}
\put(8,3){\circle{.2}}
\put(8.15,3.15){$O'$}

\put(5,4){\circle{.1}}

\put(4.6,4){$P$}
\qbezier(5,4)(11,6)(11,1)
\put(5,4){\line(2,-3){2}}
\put(5,4){\line(3,-2){3}}

\thicklines
\put(5,4){\line(1,-1){3}}

\thinlines

\put(-0.25,0.75){\framebox(5.5,1.5){}}
\put(-.1,1.8){$E_7$}

\put(6.75,-.25){\framebox(3.5,2.5){}}
\put(6.85,-.05){$I_2^*$}


\end{picture}
\caption{Divisors of type $I_2^*$ and $E_7$}
\label{Fig:fib2}
\end{figure}

On the second fibration, $P$ induces a multisection of degree $N-1$.
In order to find the section $P'$ associated to this multisection,
we subtract suitable elements from the trivial lattice (fiber components and the zero section).
For this, we fix the zero section $O'$ and the two-torsion section $R'$ as depicted.
%
%
In the present situation, we find the following intersection behaviour of $P'$ (which will also be sketched in Figure \ref{Fig:fib3}):
\begin{enumerate}
\item
$P$ meets only the identity component of $III^*$ (the one missing in the figure which thus also meets $O'$) with multiplicity $N-1$.
\item
On the $I_2$ fibers, $P$ meets the non-identity components (again missing in the figure) with multiplicity $N-1$.
Subtracting $(N-1)/2$ times the identity component $C_0^i$ ($i=1,2,3$),
we obtain a divisor meeting only the identity component with multiplicity $N-1$.
\item
On the $I_2^*$ fiber, $P$ meets the first double component with multiplicity $(N-3)/2$.
Subtracting the identity component $C_0^0$ with the same multiplicity, we obtain a divisor that meets only the identity component (multiplicity $N-2$) and the near simple component (multiplicity $1$).
\end{enumerate}
By adding suitable multiples of $O'$ and the general fiber $F'$,
we obtain a divisor $D'$ with $D'^2=-2$ that meets each fiber in exactly one point:
\[
D'=P - \frac{N-1}2(C_0^1+C_0^2+C_0^3) - \frac{N-3}2 C_0^0 - (N-2) O' + \frac{N-1}2 F'.
\]
Since $D'\equiv P$ modulo the trivial lattice,
$D'$ represents a section $P'$ only meeting the $I_2^*$ fiber in a non-identity component (near simple).
Since $P'.O'=(N-3)/2$, we find indeed that $P'$ has height $N$.

\subsection{Third elliptic fibration}

We continue with another elliptic fibration.
We extract a new divisor of Kodaira type $I_2^*$ along similar lines as in \ref{ss:fib2}.
Namely we combine the first two simple and double components of the old $I_2^*$ fiber with zero 
section and identity components of the $III^*$ fiber and one of the $I_2$ fibers.
In the orthogonal complement we find another $D_6$ (from $III^*$) and four $A_1$'s (the far simple components of the original $I_2^*$ and the non-identity components of the two avoided $I_2$'s).
Then the original two-torsion section $R'$ (which will be omitted in the next figure to simplify the presentation) 
and the two remaining components of the original fibers of type $I_2^*$ and $III^*$ serve as sections (zero and twice two-torsion).
Thus we find indeed the third fibration from \ref{ss:Km} as depicted in Figure \ref{Fig:fib3}.

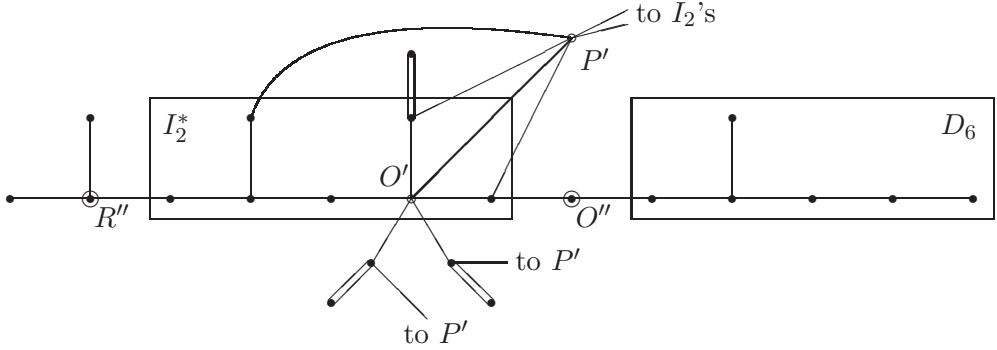
\begin{figure}[ht!]
\setlength{\unitlength}{.42in}
\begin{picture}(13,4.3)(-.15,-.7)
\thinlines

\put(5,1){\circle{.1}}
\put(4.6,1.15){$O'$}

\put(7.05,.65){$O''$}

\put(1.05,.65){$R''$}

\multiput(0,1)(1,0){5}{\circle*{.1}}

\multiput(6,1)(1,0){7}{\circle*{.1}}

\put(0,1){\line(1,0){12}}

\multiput(1,2)(2,0){3}{\circle*{.1}}
\multiput(1,2)(2,0){3}{\line(0,-1){1}}

\put(9,2){\circle*{.1}}
\put(9,1){\line(0,1){1}}

\multiput(4.5,0.2)(1,0){2}{\circle*{.1}}
\multiput(4,-.3)(2,0){2}{\circle*{.1}}

\multiput(4.47,.23)(.06,-.06){2}{\line(-1,-1){.5}}
\multiput(5.47,.17)(.06,.06){2}{\line(1,-1){.5}}

\put(5,1){\line(3,-5){.5}}
\put(5,1){\line(-3,-5){.5}}

\put(5,2.8){\circle*{.1}}
\multiput(4.96,2.8)(.08,0){2}{\line(0,-1){.8}}


%

\put(1,1){\circle{.2}}
\put(7,1){\circle{.2}}

\put(7,3){\circle{.1}}

\put(7.1,2.6){$P'$}
\qbezier(7,3)(3.5,3.5)(3,2)
\put(7,3){\line(-2,-1){2}}
\put(7,3){\line(-1,-2){1}}

\thicklines
\put(7,3){\line(-1,-1){2}}

\thinlines

\put(7,3){\line(4,1){.7}}
\put(7,3){\line(2,1){.7}}
\put(7.8,3.2){to $I_2$'s}

\put(5.5,.2){\line(1,0){.7}}
\put(6.3,.1){to $P'$}
\put(4.5,.2){\line(1,-1){.7}}
\put(4.9,-.8){to $P'$}

\put(1.75,0.75){\framebox(4.5,1.5){}}
\put(1.9,1.8){$I_2^*$}

\put(7.75,.75){\framebox(4.5,1.5){}}
\put(11.6,1.8){$D_6$}


\end{picture}
\caption{Divisors of type $I_2^*$ and $D_6$}
\label{Fig:fib3}
\end{figure}

It remains to associate a section $P''$ to the multisection induced by $P'$.
Here the multisection degree is $N$.
We choose the zero section $O''$ as indicated in the figure and denote  by $R''$ the two-torsion section which is a component of the original $I_2^*$ fiber.
Along the same lines as in \ref{ss:fib2},
one finds $P''$ meeting the $I_2$ fibers in identity components
and the $I_2^*$ fibers 
in the same far components  that $R''$ meets.
Since $P''.O''=(N-1)/2$, $P''$ has height $N$ as required.


\begin{rem}
\label{rem:field}
In this step we possibly have to extend the base field 
as we single out one of the three fibers of type $I_2$ 
(which correspond to the non-identity simple components of the original $I_0^*$ fibers 
where $P$ meets the identity component).
This extension is the reason why we cannot simply parametrise the family 
by $X_0^+(N)$ (see \ref{ss:moduli} for the case $N=3$).
\end{rem}


\subsection{$\Z/2\Z$ quotient}\label{ss: Z2 quotient}

We want to quotient out by the two-torsion section that meets both $I_2^*$ fibers in the near simple component
and hence all $I_2$ fibers in non-identity components.
In terms of the group law,  this section is $R'+R''$.
The quotient results in a new K3 surface $W_N$ with an elliptic fibration with only two reducible fibers, each of type $I_4^*$, and two-torsion in $\MW$. 
In a few steps we shall see that $W_N$ exactly realises a surface $X_N$ as in Theorem \ref{thm}.

\begin{Lemma}
$W_N$ is a quadratic base change of a rational elliptic surface.
\end{Lemma}
\begin{proof}
The according statement holds true for the previous fibration due to the singular fibers and full two-torsion.
But then the sections respect the base change property, and so does the quotient.
\end{proof}

\begin{Corollary}\label{cor: quotient YN is XN}
$W_N$ is a member of the Barth--Peters family as in \ref{ss:222}.
We have 
\[
T(W_N)=U(2)+\langle 2N\rangle,\;\;\; \NS(W_N)=U(2)+2E_8(-1)+\langle-2N\rangle.
\]
\end{Corollary}

\begin{proof}
The section $P''$ on the third elliptic fibration on $\Km$ induces a section $Q$ of height $2N$ on $W_N$.
Here $Q$ meets both $I_4^*$ in a far simple component.
By construction, the two-torsion section meets the same components,
so their sum $R$ is orthogonal to the two summands of $D_8(-1)$ in the trivial lattice corresponding the $I_4^*$ fibers.
Thus we find the following sublattice $L$ of $\NS(W_N)$:
\begin{eqnarray}
\label{eq:NS-x}
L = U + (2D_8(-1)+\Z/2\Z) + \langle-2N\rangle.
\end{eqnarray}
Assume that $L\neq \NS(W_N)$.
Then there is a divisible section in $N$, i.e. either $Q$ or $R$ is divisible.
But since they are related to $\Km$ by a $2$-isogeny, these sections could only be $2$-divisible.
However, this is impossible in the present situation
since it would result in a non-integer height $N/2$ while all correction  terms in the height formula are integers
(since singular fibers of type $I_4^*$ have only integer correction terms).

We conclude that $L=\NS(W_N)$ 
and immediately find the claimed representations for $T(W_N)$ and $\NS(W_N)$.
In particular this implies that $W_N\in\mX$. More precisely the surface $W_N$ is a general member of the family $\mX_N$ and so it is the surface called $X_N$ in Section \ref{ss:X_N}
\end{proof}

\subsection{Geometric proof of Theorem \ref{thm}}

We claim that $X_N\in\mX_N$ is a K3 surface proving Theorem \ref{thm}.
Note that $X_N$ admits a rational map of degree $2$ to $\Km(E\times E')$ 
given by the $2$-isogeny with the third elliptic fibration on $\Km(E\times E')$.

The Enriques involution $\tau$ on the Barth--Peters family 
descends to the special member $X_N$.
On the above fibration with two $I_4^*$ fibers, it is given by deck transformation of the quadratic base change composed with translation by the two-torsion section (cf.~\cite{HS}).
The invariant sublattice is contained in the trivial lattice.
By \eqref{eq:NS-x} the section $R$ is associated to a summand $D$ in $\NS(X_N)$ that is orthogonal to the trivial lattice.
Thus $D$ is anti-invariant for $\tau^*$.
As $D^2=-2N$, we find that $\pi^* \Br(X_N/\tau) = \{0\}$ by Theorem \ref{thm:Beau}.
\qed

\subsection{CM}
\label{ss:comments}

Throughout this chapter we have assumed that the isogenous elliptic curves $E, E'$ do not have CM.
This assumption serves to rule out the special members with $\rho=20$, i.e.~the singular K3 surfaces in the family.
The assumption is not strictly necessary as it only serves two minor purposes: 
First to exclude those singular K3 surfaces 
where the involution fails to be fixed point free 
or the singular fibers degenerate (as seen in \ref{ss:spec});
secondly to ensure that the K3 surfaces are indeed not Kummer as stated in Theorem \ref{thm}.
In practice, it will always suffice for our purposes to exclude a finite number of CM-points
(see Proposition \ref{prop:spec} for the family with $N=3$).

\subsection{}
\label{ss:2nd}

It is instructive to consider a second way to derive the second and  third elliptic fibration from \ref{ss:Km}.
Namely when extracting the second elliptic fibration,
we could opt to include the component met by $P$ in the divisor of Kodaira type $I_2^*$ 
(possibly at the cost of another extension of the base field).
The overall construction goes through as before,
but in the end the section $P''$ induced by $P$ on the third elliptic fibration shows a different intersection behaviour than before (even up to addition of two-torsion sections).
On the $\Z/2\Z$ quotient $W_N$ we still obtain a section $Q$ of height $2N$, but this time meeting only one $I_4^*$ fiber at a far simple component. 
Thus one finds the transcendental lattice $T(W_N)=U+\langle 8N\rangle$.

\section{Elliptic fibrations and moduli of the family $\mX_3$}
\label{s:N=3}

In sections \ref{s:lat} and \ref{s:geom} we considered generally complex K3 surfaces $X_N$ with 
$\NS(X_N) \simeq U(2) + 2E_8(-1) + \langle -2N\rangle$
for odd $N>1$.
We derived an Enriques involution geometrically as well as lattice theoretically
and showed that the Brauer group pulls back identically to zero.
Here we consider one of these families in detail,
the family $\mX_3$ such that 
\begin{equation}\label{eq:NS} \NS(\mX_3)\simeq U(2) + 2E_8(-1) + \langle -6\rangle.\end{equation}
For this family we give an explicit equation (defined over $\Q$) 
answering the problem posed in Section \ref{ss: Beauville's problem}.
We describe the Hodge structure (given by a pair of 3-isogenous elliptic curve $E$ and $E'$) 
and its specialisations (related to the complex multiplication on the elliptic curves $E$ and $E'$).
We hope that its analysis will both illustrate our methods and give the reader an idea how the constructions can be carried out explicitly.
\subsection{}

In \ref{ss: Z2 quotient} the surfaces $X_N$ are constructed as quotients of known Kummer surfaces, 
but without explicit equations. 
In order to find an explicit equation for the family $\mX_3$,  
we will exhibit a convenient jacobian elliptic fibration on $\mX_3$ 
(which is not among the ones coming from the geometric construction of Section \ref{ss:X_N}).
In the first instance, this amounts to writing $\NS(\mX_3)$ as an orthogonal sum of the hyperbolic plane $U$
and an even negative-definite lattice $L$. 
Preferably $L$ is a root lattice, since then the Mordell-Weil group of the elliptic fibration is finite (cf.~\ref{ss:ell}).
We will proceed in two steps related to the isomorphisms
\begin{eqnarray}
\label{eq:1}
\NS(\mX_3) & \simeq & U+D_8(-1)+E_8(-1) + \langle-6\rangle\\
\label{eq:2} & \simeq & U + D_8(-1) + E_7(-1) + A_2(-1).
\end{eqnarray} 
A direct computation shows that the lattices in \eqref{eq:NS}, \eqref{eq:1} and \eqref{eq:2} have the same signature, the same discriminant group, 
and the same discriminant form.  By Proposition \ref{prop: uniquely of a lattice in primitve lattice, condition on discriminant} they are isometric.
The elliptic fibration corresponding to the decomposition \eqref{eq:2} of $\NS(\mX_3)$ is particularly convenient since it only involves $U$ and root lattices.
By \ref{ss:ell} the latter correspond to reducible singular fibers of type $I_4^*, III^*, I_3$ (or $IV$ a priori).
In particular, the last elliptic fibration has no sections other than the zero section.

\subsection{The elliptic fibration $[I_4^*, III^*, I_3]$ on $\mX_3$}
We explain how to find a model of the last elliptic fibration \eqref{eq:2}.
Consider the quadratic twist at the non-reduced fibers which replaces them by fibers of type $I_4, III$.
This results in a family of rational elliptic surfaces $\fS$ with configuration of singular fibers [1,1,3,4,III].
Let $k$ be any field of characteristic different from $2$.
In extended Weierstrass form, the family of rational elliptic surfaces can easily be parametrised over $k(r)$ as
\[
\fS:\;\;\; y^2 = x^3- t (r^2 t-1-2 r) x^2-2 (t+1) t r (r t-1) x-(t+1)^2 t^2 r^2.
\]
As required the given model has the following reducible singular fibers: 
$$
\begin{array}{ccc}
\hline
III & I_3 & I_4\\
\hline
0 & -1 & \infty\\
\hline 
\end{array}
$$
Note that the general member of the family $\fS$ has Mordell-Weil rank two by the Shioda-Tate formula
\cite[Cor.~5.3]{ShMW}.
These sections are not preserved under the quadratic twist.
The family $\mX_3$ with elliptic fibration corresponding to the decomposition \eqref{eq:2} 
is recovered by a quadratic twist at $0$ and $\infty$
(i.e.~the fibrations become isomorphic over $k(\sqrt t)$):
\begin{align}
\label{eq:mX3}
\mX_3:\;\;\; y^2 = x^3- t^2 (r^2 t-1-2 r) x^2-2 (t+1) t^3 r (r t-1) x-(t+1)^2 t^5 r^2.
\end{align}
Here a general member of the family $\mX_3$ has $\rho=19$ and $\NS$ given as above with $\MW=\{O\}$.
Our next aim is to write down the Enriques involution explicitly
and to give the anti-invariant divisor.
Later we will study the parametrising curve and special members of the family.

\subsection{The elliptic fibration $[I_4^*,II^*]$ on $\mX_3$}
In order to find the elliptic fibration on $\mX_3$ corresponding to \eqref{eq:1}, it suffices to determine 
a suitable divisor of Kodaira type $II^*$.
In the present situation,
this divisor is extracted from the fiber of type $III^*$ extended by zero section and identity component of the fiber of type $I_3$.
In $\NS(\mX_3)$, this leaves the orthogonal summand $D_8(-1)$ formed by the non-identity components of the $I_4^*$ fiber.
The two other components of the $I_3$ fiber serve as zero section on the one hand and section of height $6$ on the other.
We sketch these $(-2)$-curves in the following diagram:

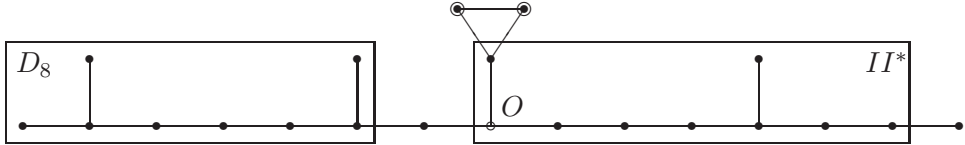
\begin{figure}[ht!]
\setlength{\unitlength}{.35in}
\begin{picture}(14,2.35)(-.15,0.65)
\thinlines

\put(7,1){\circle{.1}}
\put(7.15,1.15){$O$}


\multiput(0,1)(1,0){7}{\circle*{.1}}

\multiput(8,1)(1,0){7}{\circle*{.1}}

\put(0,1){\line(1,0){14}}

\multiput(1,2)(4,0){2}{\circle*{.1}}
\multiput(1,2)(4,0){2}{\line(0,-1){1}}

\multiput(7,2)(4,0){2}{\circle*{.1}}
\multiput(7,2)(4,0){2}{\line(0,-1){1}}

\put(6.5,2.75){\circle*{.1}}
\put(7.5,2.75){\circle*{.1}}
\put(6.5,2.75){\line(1,0){1}}
\put(6.5,2.75){\line(2,-3){.5}}
\put(7.5,2.75){\line(-2,-3){.5}}

\put(6.5,2.75){\circle{.2}}
\put(7.5,2.75){\circle{.2}}

\put(-.25,0.75){\framebox(5.5,1.5){}}
\put(-.1,1.8){$D_8$}

\put(6.75,.75){\framebox(6.5,1.5){}}
\put(12.6,1.8){$II^*$}


\end{picture}
\caption{Divisors of type $II^*$ and $D_8$}
\label{Fig:fib0}
\end{figure}

In the terminology of \cite{KS} we shall work with the following elliptic parameter
with respect to the equation \eqref{eq:mX3}:
\[
u=(x-rt^2(t+1)/2)/(t^3(t+1)).
\] 
After some variable transformations, one obtains the Weierstrass form
\begin{eqnarray}
\label{eq:(4)}
\mX_3: y^2  & = & 
t^3+2 u (r^3-8 r u-4 u) t^2+16 u^4 (1-4 r+2 r^2) t+128 r u^7
\end{eqnarray}
with reducible fibers of type $II^*$ at $u=\infty$ and $I_4^*$ at $u=0$.
The section of height $6$ is given in terms of its $t$-coordinate as
\begin{eqnarray*}
(-32 u^5+(64 r^2+336 r+128) u^4
+(-32 r^4-320 r^3-720 r^2-192 r-128) u^3\\
+8 r (6 r^4+32 r^3+21 r^2-20 r+8) u^2
-2 r^3 (12 r^3+24 r^2-27 r+8) u+r^5 (2 r-1)^2)\\
\times \frac 1{16} (r-2u)/(r^2-2 r u-r-2 u)^2
\end{eqnarray*}
Thus we have indeed found the elliptic fibration corresponding to \eqref{eq:1}.
From this, one can derive a double quartic model 
associated to the decomposition \eqref{eq:NS} of $\NS(\mX_3)$ 
after adjoining a square root via $r=(1-q^2)/4$.


\subsection{Enriques involution and the elliptic fibration $[2III^*, 2I_2]$ on $\mX_3$}
To exhibit the specified Enriques involution on the family $\mX_3$ explicitly, we use the Barth-Peters family for which we have worked out the Enriques involution
in \ref{ss:cti}.


To find the Enriques involution on $\mX_3$, 
it suffices to exhibit an elliptic fibration with two reducible fibers of type $III^*$ and $I_2$ each and two-torsion in $\MW$.
Then the Mordell-Weil lattice of a general member will have rank one and a generator of height $6$.
(The two-torsion condition is crucial since the family $\mX_3$ does also admit an elliptic fibration with the same singular fibers, but without torsion in $\MW$, so in that case $\MWL=\langle 3/2\rangle$).

We work with the fibration on $\mX_3$ corresponding to \eqref{eq:1}.
In terms of the model in \eqref{eq:(4)}, 
the elliptic parameter $v=u/t^3$
extract a divisor of type $III^*$ from components of the $I_4^*$ fiber extended by zero section and identity component of the $II^*$ fiber.
The adjacent fiber components form the new zero and two-torsion section;
the remaining fiber components form one  divisor of  type $E_7$
and  two root lattices of type $A_1$.
We sketch these rational curves in Figure \ref{Fig:BP0}
where we only omit the additional $\MW$ generator of height $6$ for simplicity.

\begin{figure}[ht!]
\setlength{\unitlength}{.33in}
\begin{picture}(16,1.8)(0,-.3)

\thinlines

\multiput(0,0)(1,0){7}{\circle*{.1}}
\multiput(1,1)(4,0){2}{\circle*{.1}}
\multiput(8,0)(1,0){8}{\circle*{.1}}
\put(7,0){\circle{.1}}
\put(13,1){\circle*{.1}}
%
\put(0,0){\line(1,0){15}}
\put(1,0){\line(0,1){1}}
\put(5,0){\line(0,1){1}}
\put(13,0){\line(0,1){1}}
%
%
\put(7.05,.15){$O$}

%
%
%
\put(1,0){\circle{.2}}
\put(9,0){\circle{.2}}
%
%
%
%
%
%
%
%
\put(1.75,-.25){\framebox(6.5,1.5){}}
\put(1.9,.75){$III^*$}
\put(9.75,-.25){\framebox(5.5,1.5){}}
\put(14.6,.75){$E_7$}
%
%

\end{picture}
\caption{Divisors of type $III^*$ and $E_7$ vs.~$II^*$'s and $I_4^*$}
\label{Fig:BP0}
\end{figure}
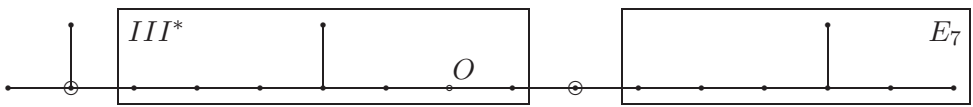

In suitable coordinates, we obtain the Weierstrass form
\begin{eqnarray}
\label{eq:BP'}
\;\;\;\;\;\;\;
\mX_3: \;\;\; y^2 = x (x^2 -8 t^2 (1+2 r) x + 2 t^3 (64 r+8 t-32 r t+16 r^2 t+t^2 r^3)).
\end{eqnarray}
The two-torsion section is $(0,0)$ as before.
The deck transformation $\imath$ for the quadratic base change $f$ is given by
\[
\imath: \, (x,y,t) \mapsto (\alpha^2 x/t^4, \alpha^3 y/t^6, \alpha/t),\;\;\; \alpha=64/r^2.
\]
The quotient by the deck transformation $\imath$ is an extremal rational elliptic surface with $\MW=\Z/2\Z$.
It follows that the section $P$ of height $6$ is anti-invariant for $\imath^*$.
Thus it is induced by a section $P'$ of height $3$ on the quotient of $\mX_3$ by the Nikulin involution $(-id) \circ \imath$.
Thanks to the low height and the presence of two-torsion, the section $P'$ is not hard to find.
Here we only give $P$ in terms of its $x$-coordinate:
\[
\frac 1{256}
\dfrac{(t^2 r^2+64+16 r t-32 t)^2 (r t-8)^2}{(r t+8)^2}
\]
As required, $P$ and $O$ intersect exactly at one of the ramification points of the quadratic base change, $t=-8/r$ (so that $P\cdot O=1$) while $P$ does not meet any fiber at a non-identity component.
An anti-invariant divisor on $\mX_3$ for the induced action of the Enriques involution $\tau$ (
composition of $\imath$ and translation by the two-torsion section $(0,0)$) is then given as
\begin{eqnarray}
\label{eq:phi(P)}
\varphi(P) = P-O-3F,\;\;\; \varphi(P)^2=-6.
\end{eqnarray}
This can be seen as follows: 
the Enriques lattice $U(2)+E_8(-2)$ embeds primitively into $\NS(\mX)$.
On $\mX_3$ where we have additional sections, this specialises to a primitive embedding into the trivial lattice of the given elliptic fibration.
By definition, 
$\varphi(P)$ is orthogonal to the trivial lattice of the elliptic fibration 
(as in the theory of Mordell-Weil lattices).
Hence it is anti-invariant for $\tau^*$
(a fact that can also be checked explicitly with Mordell-Weil lattices as in \cite{HS}).

\subsection{Moduli}
\label{ss:moduli}

In order to determine the moduli curve of the family $\mX_3$,
we can argue with the Kummer structure of the fibration on $\mX_3$  
with two fibers of type $I_4^*$ as in \ref{ss:Km-str}, \ref{ss:ec}.
Thus we find a relation to a product of elliptic curves.
Since the Picard number is generically $19$,
these elliptic curves ought to be isogenous.
We will relate them to
the modular curve $X^*(6)=X_0(6)/\langle w_2, w_3\rangle$ where we divide out $X_0(6)$
by all Fricke involutions.
This curve parametrises elliptic curves over biquadratic extensions of $\Q$ with prescribed isogenies to their Galois conjugates.
Details (mostly in the context of $\Q$-curves) can be found in \cite{Q}.
There a Hauptmodul $a$ for $X^*(6)$ is fixed.

\begin{Theorem}
\label{thm:6}
The family $\mX_3$ is parametrised by $X^*(6)$.
The parameter $r$ is related to the Hauptmodul $a$ of $X^*(6)$ by
$a=-2(r+2)/(4r-1)$.
\end{Theorem}

\begin{proof}
For the elliptic curves parametrised by $X^*(6)$, a Weierstrass form is given in \cite{Q}.
In particular, we obtain the j-invariants of these elliptic curves.
Since the field of definition is $\Q(\sqrt{a},\sqrt{2a+1})$,
it is actually more convenient to write
\begin{eqnarray}
\label{eq:a}
a=a(t) = \left(\frac{2t}{2t^2-1}\right)^2
\end{eqnarray}
which makes both square roots rational in $t$.
Then the $j$-invariants are represented by
\begin{eqnarray}
\label{eq:j}
j(t)=\frac{6912 (5 t^3+6 t^2-2)^3 t^3}{(2 t-1) (t+1)^2 (2 t+1)^3 (t-1)^6}
\end{eqnarray}
up to conjugation.
We will show that these j-invariants coincide with those coming from the Kummer structure on $\mX_3$.
This suffices to prove the theorem.

Arguing as for the full Barth-Peters family in \ref{ss:Km-str}, \ref{ss:ec}, only specialised to $\mX_3$,
we find the following relations in terms of the parameter $q=\sqrt{1-4r}$:
\[
j\cdot j'=4096 \frac{(q-3)^3 (25 q^3+15 q^2+3 q-51)^3}{(q+1)^8 (q-1)^4},
\]
\[
j+j' = 128 \frac{(125 q^6+800 q^5-715 q^4-3400 q^3+7511 q^2-5464 q+1399) (q-3)^3}{(q-1)^3 (q+1)^6}.
\]
We now proceed in three steps.
First we employ the modularity methods of point counting and lifting from \cite{ES} 
to find numerically members over $\Q$ of the family $\mX_3$ with $\rho=20$.
We find CM-values of $r$ as given in Table \ref{tab-CM}.
Then we try to match these values with the CM-points of $X^*(6)$.
This leads exactly to the given relation between the parameters $r$ and $a$.
Note that this relation is still conjectural, but it gives $q=3/\sqrt{2a+1}$.
Next we insert \eqref{eq:a} for $a$ which leads to
\[
q=\pm \frac{3 (2 t^2-1)}{1+2 t^2}
\]
The negative sign choice gives precisely the j-invariant \eqref{eq:j} and its conjugate by $t\mapsto -t$ as  solutions to the system of equations for $j, j'$ coming from the Kummer structure.
Thus $\mX_3$ is indeed parametrised by $X^*(6)$,
and the conjectural relation between $r$ and $a$ holds true as claimed.
\end{proof}



\subsection{Specialisations}
\label{ss:spec}

In \cite{HS}, an explicit singular K3 surface $X$ over $\Q$ was exhibited
with an Enriques involution $\tau$ over $\Q$ and a $\tau^*$-anti-invariant divisor $D$ over $\Q(\sqrt{-3})$ with $D^2=-6$.
This surface is given abstractly by the transcendental lattice $T(X)=\langle 4\rangle +\langle 6\rangle$.
In the above family $\mX_3$, it can be located by degenerating the two fibers of type $I_2$ of the elliptic fibration \eqref{eq:BP'} to one fiber of type $I_4$.
Actually there are two ways to achieve this degeneration: by specialising $r=1/2$ and $r=1/4$.
We will now distinguish these two cases.

In the first case, the section $P$ degenerates as well in the following sense:
its $x$-coordinate attains a double root at the $I_4$ fiber at $t=16$.
In consequence, the height drops to $h(P)=5$.
Thus the discriminant of the specialisation is $-20$.
The transcendental lattice of the special member has the following quadratic form and specialisation embedding:
\[
\begin{pmatrix} 4 & 2\\2 & 6\end{pmatrix} \hookrightarrow U(2) + \langle 6\rangle.
\]

If $r=1/4$, then the two $I_2$ fibers are merged without the section $P$ degenerating.
In terms of the elliptic fibration corresponding to the decomposition \eqref{eq:2},
this is explained by the degeneration $D_8\hookrightarrow D_9$.
Hence the special member $X$ has
\[
\NS(X) = U + A_2(-1)+E_7(-1)+D_9(-1)
\]
with transcendental lattice $T(X)=\langle 4\rangle +\langle 6\rangle$ as in \cite{HS}.
In the model \eqref{eq:BP'}, the splitting field of the $I_4$ fiber is $\Q(\sqrt{-3})$.
One can easily work out an isomorphism over $\Q(\sqrt{-3})$ with the model in \cite{HS}.
Note that for the above model, the invariant subspace of $\NS(X)$ under ${\tau^*}$ is fully Galois invariant.

Through the rational CM-points and cusps of the modular curve $X^*(6)$,
we find all other specialisations over $\Q$ with $\rho=20$.
Together with the previous two specialisations
and all corresponding discriminants $d$, we collect 
the CM-points (or rather their inverses) in Table \ref{tab-CM}.

\begin{table}[ht!]
$$
\begin{array}{c|ccccccccccccccccc}
d & -12& -15& -20& -24& -36& -48& -60& -72& -84\\ 
\hline
r^{-1} & -1/2& 5/8& 2& 4& -2& 25/4& -49/8& 12& -14\\
\end{array}
\hspace{.61cm}
$$
$$
\begin{array}{c|cccccccccccccccc}
d & -120 & -132& -168& -228& -312& -372& -408& -708\\
\hline
r^{-1} & 40 & -50& 112& -338& 1300& -3038& 4900& 
-140450\\
\end{array}
$$
\caption{CM-points of $\mX$ over $\Q$}
\label{tab-CM}
\end{table}
\begin{Proposition}
\label{prop:spec}
Let $X$ be a special member of $\mX_3$ at $r\neq 0$.
Then the Enriques involution on $\mX_3$ specialises without fixed points to $X$
if and only if $r\neq -2$.
Moreover, if $r\neq -2$, there is a divisor $D$ on $X$
such that $\tau^*D=-D$ in $\NS(X)$ and $D^2=-6$.
In particular, $\pi^*(\Br(X/\tau))=\{0\}$.
\end{Proposition}

\begin{proof}
For the specialisation at $r=-2$, there is an additional singular fiber of type $I_2$ in both elliptic fibrations \eqref{eq:2}, \eqref{eq:BP'}. 
At the same time the former fibration attains a two-torsion section while for the latter fibration, the section $P$ becomes two-divisible.
By construction, the additional $I_2$ fiber is necessarily fixed under the deck transformation above
(i.e.~the base change $f$ ramifies there).
Since the two-torsion section meets the identity component of the additional $I_2$ fiber,
the Enriques involution on $\mX$ does not specialise to a fixed point free involution at $r=-2$.

On all other K3 specialisations, the fixed fibers of the deck transformation are either smooth 
or of type $I_4$ as above.
Hence the Enriques involution stays fixed point free under specialisation.
This proves the first claim.

For the divisor $D$, we can take $D=\varphi(P)=P-O-3F$ as in \eqref{eq:phi(P)}
in all non-degenerate cases.
The same divisor works also at $r=1/4$ as $P$ meets the degenerate $I_4$ fiber at the identity component (and the height is unchanged).

In contrast, at $r=1/2$, the height of $P$ degenerates to $5$
as
$P$ meets the component $\Theta_2$ of the $I_4$ fiber (numbered cyclically).
Thus we still have $P.O=1$, 
but the section $P'=[-P+(0,0)]$ meets the $I_4$ fiber at the identity component
(as opposed to generically intersecting both $I_2$ fibers at the \emph{non}-identity components).
Hence the intersection number $P'.O$ degenerates from $3$ to $2$ at $r=1/2$.
We claim that the following modification of \eqref{eq:phi(P)} suffices to satisfy the conditions of the proposition:
\[
D:=P-O-3F+\Theta_1+\Theta_2.
\]
The intersection number $D^2=-6$ is easily verified.
For the anti-invariance,
we note that $\tau$ rotates the $I_4$ fiber.
Hence
\[
\tau^* D = P' -(0,0) - 3F + \underbrace{\Theta_3 + \Theta_0}_{= F - \Theta_1 - \Theta_2}
\]
To prove that $D+\tau^*D=0$ in $\NS(X)$,
we then only need to verify 
that the sum is orthogonal to the trivial lattice (fiber components and zero section)
and that it gives zero in $\MW(X)$.

With the divisor $D$ at hand,
the final claim of the proposition follows directly from Theorem \ref{thm:Beau}.
\end{proof}


\section{Kummer Surfaces}
\label{s:Km}

Our results about Enriques involutions and Brauer groups so far 
have exclusively concerned K3 surfaces that are generally not Kummer.
Here we want to extend this approach to Kummer surfaces
which turn up very naturally for the Barth--Peters family in the realm of Shioda--Inose structures.

\subsection{Shioda--Inose structure on the Barth--Peters family}
\label{ss:X-Y}

%
%
%
%
%
%

Recall that the Barth--Peters family $\mX$ 
admits an elliptic fibration with $I_{16}$ fiber and two-torsion section.
Translation by this section induces a Morrison--Nikulin involution $\jmath$ on $\mX$,
so that the desingularisation of $\mX/\jmath$ gives a family of Kummer surfaces
that we denote by $\mY$.
By standard formulas, we obtained as induced elliptic fibration
\begin{eqnarray}
\label{eq:WF-Y}
\mY:\;\;\; 
y^2=x(x^2-2a(t)x+(a^2(t)-4)).
\end{eqnarray}
%
Generally this has a fiber of type $I_8$ at $\infty$ and 8 fibers of type $I_2$.
There is full two-torsion consisting of the sections $(0,0), (a\pm 2,0)$.
A general member $Y\in\mY$ has transcendental lattice $T(Y)=U(2)+U(4)$.

\begin{Lemma}
\label{Lem:Y}
Let $Y$ be a member of the family $\mY$ with $\rho(Y)=18$.
Then $\NS(Y)$ is generated by torsion sections,  fiber components and a section of height $1/2$.
In particular $\MWL(Y)=[1/2]$.
Pulling back the infinite section to the quotient $X\in\mX$, we obtain a $\MWL$--generator of $X$.
\end{Lemma}

\begin{proof}
By construction, $Y$ is a base change of a rational elliptic surface.
Namely this property carries over directly from $\mX$ as $a(t)$ is in fact quadratic in $t$.
Due to the singular fibers ($I_4$ and four times $I_2$),
the rational elliptic surface has Mordell-Weil rank one.
The generator is a section of height $1/4$,
so the pull-back has height $1/2$ on $Y$.
Comparing discriminants, we verify that 
this infinite section together with torsion sections and fiber components generates $\NS(Y)$.
The final claim follows directly from the heights.
\end{proof}
%

\subsection{Involutions}
\label{ss:Y-inv}

Since the family $\mY$ is a base change of a family of rational elliptic surfaces,
we have in addition to the translations by two-torsion sections and the hyperelliptic involution
the deck transformation $\imath$ acting as a non-symplectic involution on $\mY$.
As in \cite{HS} this allows us to derive Enriques involutions on $\mY$ 
(outside some subfamily of codimension one).
Since the two-torsion sections are both invariant and anti-invariant for $\imath^*$,
the composition of their translation with the deck transformation $\imath$ defines an involution $\tau$ on $\mY$.

\begin{Lemma}
\label{Lem:Y-tau}
\label{lem1}
Generally the involution $\tau$ is an Enriques involution
if and only if the two-torsion section involved is not $(0,0)$.
\end{Lemma}

\begin{proof}
We start by studying the fixed locus of the deck transformation $\imath$.
Here $\imath$ fixes the fiber of type $I_8$ at $\tau=\infty$ and 
the fiber at $\tau=0$.
The latter fiber is smooth  outside a subfamily of codimension one.
Since two-torsion sections are always disjoint from the zero section,
the composition is fixed point free if and only if the two-torsion section meets a different component of the 
$I_8$ fiber than the zero section.
This exactly rules out the given section.
\end{proof}

In summary, outside a subfamily of codimension one,
Lemma \ref{Lem:Y-tau} gives two Enriques involutions on $\mY$.

\subsection{Specialisations}

We continue this approach by specialising the family $\mX$ to the subfamily $\mX_N$ 
comprising the surfaces $X_N$ from \ref{ss:X_N} for some fixed $N\in\N$.

As before the Morrison--Nikulin involution $\jmath$ exhibits a Shioda--Inose structure.
This time it relates to abelian surfaces $A$ with $T(A)=U(2)+\langle2N\rangle$.
Thus the desingularisation $Y_N$ of $X_N/\jmath$ coincides with the Kummer surface $\Km(A)$ with
$T(Y_N) = U(4) +\langle4N\rangle$.

\begin{Lemma}
\label{lem17}
Let $Y_N$ be as above with $\rho(Y)=19$.
\begin{enumerate}
\item
If $N>1$, then the induced elliptic fibration on $Y_N$ has non-degenerate singular fibers and 
$\MW(Y_N)=(\Z/2\Z)^2\times \Z^2$ with $\MWL(Y_N)\simeq \left[\begin{array}{ll}1/2&0\\0&N\end{array}\right]$.
\item
If $N=1$, then the induced elliptic fibration on $Y_N$ has singular fibers
$I_8+I_4+6I_2$ and $\MW(Y_N)=(\Z/2\Z)^2\times \Z$ with $\MWL(Y_N)=[1/2]$.
\end{enumerate}
\end{Lemma}

\begin{proof}
We will only treat the case $N>1$.
The case $N=1$ is proved analogously.

Pull--back from $X_N$ induces a sublattice $\MWL(X_N)[2]\simeq \left[\begin{array}{ll}2&0\\0&4N\end{array}\right]$ of $\MWL(Y_N)$.
Together with torsion sections and fiber components,
this generates a sublattice of $\NS(Y_N)$ of rank $19$ and discriminant $2^{10}N$.
Since $T(Y_N)$ has discriminant $64N$, 
the index is four and has to be accounted for completely by $\MWL(Y_N)$.
But there we can only have two-divisibility due to the $2$-isogeny between $Y_N$ and $X_N$.
Hence $\MWL(Y_N) = \MWL(X_N)[1/2]$.
\end{proof}


\subsection{Enriques involutions}

As in \ref{ss:Y-inv},
the deck transformation $\imath$ composed with  translation by either two-torsion section defines an
involution $\tau$ on $Y_N$.
In the following we refer to the Weierstrass form \eqref{eq:WF-Y} specialised to $Y_N$ for the natural elliptic fibration.

\begin{lemma}
\label{lem2}
The involution $\tau$ on $Y_N$ fails to be an Enriques involution
exactly in the following two cases:
\begin{enumerate}
\item
the two-torsion section defining $\tau$ is $(0,0)$;
\item
$N=1$ and the two-torsion section defining $\tau$ is $(a\pm 2,0)$ for the sign such that $t\nmid (a\pm 2)$.
\end{enumerate}
\end{lemma}

\begin{proof}
The argument from Lemma \ref{Lem:Y-tau} rules out the first alternative.
Recall that the deck transformation $\imath$ fixes the fiber at $t=0$.
If this fiber is smooth,
then the same argument shows that $\tau$ has no fixed points for the two-torsion sections $(a\pm 2,0)$.
Presently,
this fiber is singular (type $I_4$) exactly in the degenerate case $N=1$.
Then in order to induce a fixed point free action on the fiber,
the two-torsion section has to meet a non-identity component of the $I_4$ fiber.
This is exactly the case $t\mid(a\pm 2)$.
\end{proof}

\subsection{Brauer group}
Eventually we want to compute how the Brauer group pulls back from the Enriques quotients of $Y$ and $Y_N$.
Here's the result:

\begin{Theorem}
\label{thm:Br}
Consider the K3 surfaces $Y$ and $Y_N$ with Enriques involution $\tau$
as  specified in Lemma \ref{lem1}, \ref{lem2}. Then
\begin{enumerate}
\item
If $N$ is even, then $\pi^* \Br(Y_N/\tau) = \Z/2\Z$.
The same holds true for $Y$.
\item
If $N$ is odd, then $\pi^* \Br(Y_N/\tau) = \{0\}$.
\end{enumerate}
\end{Theorem}

In particular, the above theorem implies Theorem \ref{thm2}. 
Note that $Y_N$ admits a rational map of degree four to the Kummer surface of $E\times E'$;
this is the composition of the two $2$-isogenies $Y_N\dasharrow X_N$ and $X_N\dasharrow \Km(E\times E')$ that we have exhibited before.

The proof of Theorem \ref{thm:Br} will be given below.
First we set up some notation regarding the given elliptic fibrations on $Y$ and $Y_N$.

\subsection{Set-up}\label{ss: setup}

On $Y$, we number the components of the $I_8$ fiber cyclically $\Theta_0,\hdots,\Theta_7$ so that $\Theta_0$ meets the zero section $O$.
We number the $I_2$ fibers from $1$ to $8$.
In the following, we refer to their non-identity components simply by the respective number.
Then we have the sections $U=(0,0)$ and $V, W=(a\pm 2,0)$.
Moreover there is a section $P$ of height $1/2$ obtained from the quotient rational elliptic surface $Y/\imath$ by pull-back.
Up to rearranging the fibers (and adding a two-torsion section to $P$),
we can assume the following intersection pattern where we only indicate the $I_2$ fibers met at non-identity components:

$$
\begin{array}{cccccccccccc}
& I_8 && I_2\text{'s:} & 1 & 2 & 3 & 4 & 5 & 6 & 7 & 8\\
\hline
U & \Theta_0 &&& 1 & 2 & 3 & 4 & 5 & 6 & 7 & 8\\ 
V & \Theta_4 &&&1 & 2 & 3 & 4 & &&&\\ 
W & \Theta_4&&&&&&&5 & 6 & 7 & 8\\ 
P & \Theta_2&&& 1 & 2&&& 5 & 6&&  
\end{array}
$$
From the Mordell-Weil pairing, it follows that all these sections are orthogonal on $Y$.
The deck transformation $\imath$ acts trivially on the sections of $Y$ and on the $I_8$ fiber while permuting the $I_2$ fibers as $(12)(34)(56)(78)$.

A $\Z$-basis of $\NS(Y)$ can be obtained by omitting $W$ and the non-identity components $3, 8$, say.
There are 18 divisors remaining:
13 non-identity components and $\Theta_0$ as well as the four sections $O,U,V,P$.
The Gram matrix comprising their intersection numbers
has full rank and determinant $-64$.
Thus the specified divisors form a $\Z$-basis of $\NS(Y)$  as claimed.

In terms of this $\Z$-basis, it is easy to implement the Enriques involutions $\tau$ from Lemma \ref{lem1}.
For each Enriques involution, one finds that 
\begin{eqnarray}
\label{eq:Y-anti}
\NS(Y)^{\imath^*=-1} \cong E_8(-2).
\end{eqnarray}
In particular, all anti-invariant divisors $D$ have $D^2\equiv 0\mod 4$.
Hence $\pi^*\Br(Y/\tau)=\Z/2\Z$ by Theorem \ref{thm:Beau}.
This proves Theorem \ref{thm:Br} for $Y$.
We now turn to the subfamilies $Y_N$.

\subsection{Lattice enhancement}
In this and the next section, we assume $N>1$ 
(see \ref{ss:N=1} for the case $N=1$).
By Lemma \ref{lem17} the subfamilies $Y_N$ attain an additional section $Q$ of height $N$.
We determine the fiber components met by $Q$ through the lattice enhancement construction \ref{constr}.

\begin{Lemma}
Up to renumbering, $Y_N$ has a section $Q$ of height $N>1$ that only meets the following singular fibers non-trivially:
$$
\begin{array}{lcccccccc}
N \text{odd} & 1 & 2 &&& 5 & 6 & 7 & 8\\
N \text{even} & 1 &&& 4 && 6 & 7 &
\end{array}
$$
\end{Lemma}

\begin{proof}
The surface $Y_{N}$ is obtained by specializing $Y$ as in \ref{construction: specialisation} 
by choosing the vector $v$ to be $v_N:=(1,-N,0,0)\in U(2)+ U(4)\simeq T(Y)$.
This is consistent with the  specialisation of $X$ to $X_N$.
Then  $\NS(Y_N)$ is an overlattice of index 2 of $\NS(Y)+\langle v\rangle$.
To find a generator of the full N\'eron-Severi lattice, we fix a basis within the discriminant group of $\NS(Y)$ that corresponds to the summand $U(2)$ of $T(Y)$ (as the orthogonal summand $U(4)$ is not affected by the specialisation).
In the present situation, this basis can be given as 
\[
m_1=\frac{(1467)}2, \;\;\; m_2=\frac{(2458)}2
\]
where $(1467)$ means the sum of the divisors $1$, $4$, $6$, $7$. 
Then one can easily check that $D=(v_N+m_1+Nm_2)/2$ is in $H^2(X,\Z)$.
This gives the missing generator of $\NS(Y_N)$.

It remains to find a section $Q$ corresponding to $D$.
For this we add fiber components and sections in such a way to $D$ that the resulting divisor $Q$ has self-intersection $Q^2=-2$ and meets every fiber in exactly one component:
\begin{eqnarray}
\label{eq:Q}
Q=
\begin{cases}
\frac{v_N-(125678)}2 + O + \frac{N+3}2 F & N \text{ odd},\\
\frac{v_N-(1467)}2 + O + \frac{N+2}2 F & N \text{ even}.
\end{cases}
\end{eqnarray}
In particular we read off the intersection behaviour claimed in the lemma,
and one verifies the given height.
\end{proof}

\subsection{Proof of Theorem \ref{thm:Br}, $N>1$}

It is immediate how to extend the action of the Enriques involution from $Y$ to $Y_N$:
On $\NS(Y)$ (and its image in $\NS(Y_N)$) it is known, and on $v_N$, $\tau$ acts as $-1$ since $v_N$ specialises from $T(Y)$ which sits in the anti-invariant part.
This determines the action of $\tau^*$ on $\NS(Y_N)$ completely.
For instance, consider the case where $N>1$ is odd
and $\tau$ composes $\imath$ with translation by $W$.
Then one derives
\begin{eqnarray}
\label{eq:tau-Q}
\tau^* Q = -Q -(12) + O + W + (N+1)F.
\end{eqnarray}
Independent of the parity and the two-torsion section involved in $\tau$,
we know that the section $[2Q]$ meets all fibers at their identity components.
In fact, from the description in \eqref{eq:Q} one derives
\[
[2Q] = v_N + O + 2NF.
\]
The orthogonal projection $\varphi$  with respect to the hyperbolic plane $U=\langle O,F\rangle$ gives exactly the divisor
\[
\varphi([2Q]) = [2Q] - O - 2NF = v_N. 
\]
That is, in $\MW(Y_N)$ the section $[2Q]$ corresponds exactly to $v_N$.
By definition, this divisor is orthogonal  to the whole image of $\NS(Y)$ in $\NS(Y_N)$
and anti-invariant for $\tau^*$.
Using \eqref{eq:Y-anti} we thus find the following sublattice of the anti-invariant part of $\NS(Y_N)$:
\[
\NS(Y_N)^{\tau^*=-1} \hookleftarrow \text{im}(\NS(Y)^{\tau^*=-1}) + \langle\varphi([2Q])\rangle
\cong E_8(-2) + \langle -4N\rangle.
\]
The crucial question now is whether the inclusion above is actually an equality or whether we have a proper sublattice of index two.
The index cannot be bigger since we only have to solve whether the section $Q$ itself contributes to the anti-invariant part or only its multiple $[2Q]$.
The following lemma answers this question:

\begin{Lemma}
\label{lem:div}
\begin{enumerate}
\item If $N$ is even, then $\NS(Y_N)^{\tau^*=-1} \cong E_8(-2) + \langle -4N\rangle$.
\item
If $N>1$ is odd, then
$\NS(Y_N)^{\tau^*=-1}$ is an overlattice of 
$E_8(-2) + \langle -4N\rangle$ of index two.
It is generated by an anti-invariant divisor for $\tau^*$ of self-intersection $-N-3$ or $-N-5$.
\end{enumerate}
\end{Lemma}

\begin{proof}
To prove the first case, we used a computer program to express the above $\Q$-basis of $\NS(Y_N)^{\tau^*=-1}$
in terms of the specified $\Z$-basis of $\NS(Y_N)$. 
Then it is easily verified that the given lattice is not two-divisible.

To prove the second case, it suffices to exhibit an anti-invariant divisor $D\in\NS(Y_N)$ for $\tau^*$
for each given intersection number.
Since either $-N-3$ or $-N-5$ is not congruent to zero modulo $4$,
the respective divisor cannot be contained in $\text{im}(\NS(Y)^{\tau^*=-1}) + \langle\varphi([2Q])\rangle
\cong E_8(-2) + \langle -4N\rangle$.

Here we only give these divisors for the case where $\tau$ is $\imath$ composed with translation by $W$.
The Enriques involution involving $V$ can be dealt with similarly.
Consider the following divisor classes on $Y_N$:
\begin{eqnarray*}
D_1=
Q+(1)-O-\frac{N+1}2 F
 & \Rightarrow & D_1^2=-N-3, \tau^* D_1=-D_1.\\
D_2=Q-(4)-V-\frac{N-1}2 F
& \Rightarrow & D_2^2=-N-5, \tau^* D_2=-D_2.
\end{eqnarray*}
Let us check that these divisors are anti-invariant for $\tau^*$.
For instance
\[
\tau^*D_1 = \tau^* Q + (2) - W - \frac{N+1}2 F.
\]
By \eqref{eq:tau-Q}, one immediately finds $D_1+\tau^*D_1=0$.
A simple computation gives $D_1^2=-N-3$.
Similarly one finds $D_2^2=-N-5$ and
\[
D_2 + \tau^* D_2 = O + W - U - V - (1234) + 2F.
\]
One directly checks 
that the above divisor is perpendicular to the trivial lattice of $Y_N$ (fiber components and zero section).
Moreover $D_2 + \tau^* D_2$ induces the zero section in $\MW(Y_N)$.
Hence $D_2 + \tau^* D_2=0$ in $\NS(Y_N)$.
\end{proof}

Theorem \ref{thm:Br} follows from Lemma \ref{lem:div} as a direct application of Theorem \ref{thm:Beau}.
Recall that this implies Theorem \ref{thm2} for $N>1$.
\qed

\subsection{Proof of Theorem \ref{thm:Br}, $N=1$}
\label{ss:N=1}

By Lemma \ref{lem17}, $Y_1$ admits an elliptic fibration with singular fibers $I_8+I_4+6I_2$ 
and $\MW(Y_1)=\left(\Z/2\Z\right)^2\times \Z$. 
The surface $Y_1$ is obtained as a specialisation of $Y$ as in \ref{constr} by choosing the vector $v_1\colon=(1,-1,0,0)\in U(2)+ U(4)\simeq T(Y)$ and hence the transcendental lattice of $Y_1$ is $\langle 4\rangle + U(4)$. 
Geometrically this specialisation consists of merging two fibers of type $I_2$ 
of the given elliptic fibration on $Y$, in order to obtain a fiber of type $I_4$ on $Y_1$. 
As usual its components are numbered $C_0,\hdots, C_3$.
We recall that $Y$ admits a $2\colon1$ map to a rational elliptic surface. 
Since this base change extends naturally to $Y_1$,  
the specialisation is ramified at an $I_2$ fiber of the rational elliptic surface,
and we are merging two fibers on $Y$ which are exchanged by the deck transformation $\imath$.
We shall assume that these two fibers are the $7$-th and the $8$-th fiber
(with the same notation as \ref{ss: setup}). 
The sections of $Y$ specialise to sections of $Y_1$,
so we obtain the following intersection pattern:
$$
\begin{array}{cccccccccccc}
& I_8 &I_4& I_2\text{'s:} & 1 & 2 & 3 & 4 & 5 & 6 \\
\hline
U & \Theta_0 &C_2&& 1 & 2 & 3 & 4 & 5 & 6 \\ 
V & \Theta_4 &C_0&&1 & 2 & 3 & 4 & &\\ 
W & \Theta_4&C_2&&&&&&5 & 6 \\ 
P & \Theta_2&C_0&& 1 & 2&&& 5 & 6  
\end{array}
$$
In this set-up,
the composition of the deck transformation $\imath$ with the translation by the 2-torsion section $W$ 
is an Enriques involution (see Lemma \ref{lem2}). 
Consider the divisor 
\[
D:=\Theta_4+\Theta_5+\Theta_6+\Theta_7+C_2+C_3-(24)-V+W
\]
One easily checks that $D$ is $\tau^*$-anti-invariant and $D^2=-6$. 
By Theorem \ref{thm:Beau}, this concludes the proof of Theorem \ref{thm:Br} in case $N=1$.
This completes the proof of Theorem \ref{thm2}. \qed

\subsection{A rigid example} 
\label{ss:GvG}

We conclude this paper with a 
 brief description of the unpublished example from \cite{GvG} mentioned in the introduction. 
This example comes up naturally here
as it appears as a specialisation of the Kummer surface $Y_1$ described in the previous paragraph. 
Let us specialise (as in \ref{constr}) the surface $Y_1$ choosing the vector $v$ to be $v:=(1,1,-1)\in \langle 4\rangle+ U(4)\simeq T(Y_1)$. 
The surface obtained has transcendental lattice isomorphic to $\left[\begin{array}{rr}4&0\\0&4\end{array}\right]$. 
Hence it is the Kummer surface, $\Km(E_i\times E_i)$, of the product of $E_i$ (the elliptic curve with an automorphism of order 4) with itself. 
Geometrically this specialisation consists of merging two further fibers of type $I_2$ with the fiber of type $I_4$. 
After this specialisation,
the given elliptic fibration of $Y_1$ attains  singular fibers $2I_8+4I_2$. 
As before we have to choose the fibers of type $I_2$ that we merge with the fiber of type $I_4$ 
in such a way that the $2\colon1$ map from $Y_1$ to a rational elliptic surface is preserved. 
This only allows the $I_2$ fibers $5, 6$.
The rational elliptic surface is forced to degenerate as well, attaining a second fiber of type $I_4$.
Above this fiber, the K3 surface has the degenerate fiber of type $I_8$.
Numbering its fiber components cyclically $D_0,\hdots,D_8$, 
we obtain the following intersections:
$$
\begin{array}{cccccccccccc}
& I_8 &I_8& I_2\text{'s:} & 1 & 2 & 3 & 4  \\
\hline
U & \Theta_0 &D_4&& 1 & 2 & 3 & 4  \\ 
V & \Theta_4 &D_0&&1 & 2 & 3 & 4 \\ 
W & \Theta_4&D_4&&&&& \\ 
P & \Theta_2&D_2&& 1 & 2&&  
\end{array}
$$
The infinite section $P$ on $Y_1$ becomes a 4-torsion section on $\Km(E_i\times E_i)$ 
(induced from the rational elliptic surface underneath),
 and $W$ is exactly twice $P$. 
 Hence the Mordell--Weil group of the elliptic fibration on $\Km(E_i\times E_i)$ is $\Z/4\Z\times \Z/2\Z$. 
 
The Enriques involution $\tau$ on $Y_1$ 
(composition of the deck transformation with the translation by  $W$)
specialises without fixed points to $\Km(E_i\times E_i)$. 
The following divisor is anti-invariant for $\tau^*$:
\[
D:=O-U+\Theta_4+\Theta_5+\Theta_6+\Theta_7+D_4+D_5+D_6+D_7-(13).
\] 
Since $D^2=-10$, the Brauer group of $\Km(E_i\times E_i)/\tau$ pulls back to zero by Theorem \ref{thm:Beau}. 

We note that this elliptic fibration on $\Km(E_i\times E_i)$ as well as the Enriques involution $\tau$  can be defined over $\Q$.
As a specialisation of \eqref{eq:WF-Y},
it admits the Weierstrass form
\begin{eqnarray}
\label{eq:ii}
\Km(E_i\times E_i):\;\;\; y^2 = x(x-t^4)(x-t^4+4).
\end{eqnarray}
Here $\imath(t)=-t$ and $W=(t^4,0)$.

\subsection{}
Naturally the surface $\Km(E_i\times E_i)$ is the quotient by a Morrison--Nikulin involution of a surface
$S$ specialising from $\mX_1$.
Here $S$ is obtained from \eqref{eq:ii} by the $2$-isogeny induced by the two-torsion section $(0,0)$.
In accordance with Sections \ref{s:lat} and \ref{s:geom}, 
the K3 surface $S$ admits an elliptic fibration with singular fibers $I_{16}+I_4+4I_1$ and $\MW=\Z/4\Z$.
The Morrison--Nikulin involution is the translation by the 2-torsion section.
Abstractly $S$ is given as desingularisation of
 $(E_i\times E_i)/\langle\alpha\times \alpha^3\rangle$ 
 where $\alpha$ is an order four automorphism of $E_i$.

\subsection{}
For each $N$, we constructed two different 1-dimensional families of K3 surfaces with an Enriques involutions: the families $\mX_N$ and the families $\mY_N$ related by a $2$-isogeny for fixed $N$. 
The families show a nice interplay at specialisations with $\rho=20$.
For instance the surface $\Km(E_i\times E_i)$ is both a specialisation of $\mY_1$ and of $\mX_2$.
Note, however, that the induced Enriques involutions may differ (or degenerate) as we show below.

Above we have constructed $\Km(E_i\times E_i)$ as a specialisation of $\mY_1$.
The induced Enriques involution $\tau$ had the anti-invariant divisor $D$ with $D^2=-10$.
On the other hand $\Km(E_i\times E_i)$ arises as a specialisation of $\mX_2$ as in \ref{constr} 
by choosing the vector $v$ to be $v:=(0,1,-1)\in \langle 4\rangle + U(2)$. 
The Enriques involution on $\mX_2$ induces an Enriques involution $\tau_2$ on $\Km(E_i\times E_i)$.
In this case, one finds that there is no anti-invariant divisor $D$ on $\Km(E_i\times E_i)$ 
such that $D^2\not\equiv 0\mod 4$. 
In consequence the Enriques involutions  on $\Km(E_i\times E_i)$ 
induced from $\mY_1$ and from $\mX_2$ are not conjugate in the automorphism group of $\Km(E_i\times E_i)$. 

Alternatively one can argue with the automorphism groups of the Enriques quotients.
It follows immediately from the construction 
that the Enriques surface $\Km(E_i\times E_i)/\tau$ admits 
an elliptic fibration with two double fibers of type $I_4$ 
while $\Km(E_i\times E_i)/\tau_2$ admits an elliptic fibration with one double fiber of type $I_8$. 
By Kondo's classification (see especially \cite[Table 2]{Kondo}) 
these Enriques surfaces have different (finite) automorphism groups.

\subsection*{Acknowledgements}

Most of  this work
was carried out during the authors' visits to each other's home institution
whom we thank for the great support and warm hospitality.
Particular thanks go to Bert van Geemen and Klaus Hulek.
We are grateful to the referee for many helpful comments.

\end{document}